\documentclass[12pt]{article}

\usepackage{amsmath, amsthm, amssymb,amsrefs}
\usepackage{geometry}
\usepackage{hyperref}
\usepackage{enumitem}
\usepackage{cite}
\usepackage{mathrsfs}
\usepackage{cleveref}
\theoremstyle{plain}
\newtheorem{theorem}{Theorem}[section]
\newtheorem{definition}{Definition}[subsection]
\newtheorem{lemma}[theorem]{Lemma}
\newtheorem{proposition}[theorem]{Proposition}
\newtheorem{corollary}[theorem]{Corollary}

\theoremstyle{definition}

\theoremstyle{remark}

\title{\textbf{ Uniformly bounded weight modules for Map extended Special and Map extended Hamiltonian Lie algebras}}
\author{ Pradeep Bisht and Punita Batra\thanks{Corresponding author} \\ \small Harish-Chandra Research Institute, Prayagraj. A CI of Homi Bhabha National Institute,\\ \small Chhatnag Road, Jhunsi, Prayagraj(Allahabad) 211019 India  \\ \small \texttt{pradeepbisht@.hri.res.in} and \texttt{batra@hri.res.in}} 

\date{}

\begin{document}

\maketitle

\begin{abstract}
This paper explores the irreducible uniformly bounded weight modules of map extended Special Lie algebra and map extended Hamiltonian Lie algebra under some assumption on the action of $A_{N}$. They turn out to be irreducible modules for the underlying Special and Hamiltonian Lie algebras(respectively) and the action of $A_{N}\otimes B$ is shown to be associative except for the zero-degree elements which act as scalars. It is also proved that such modules are also irreducible for the underlying extended Special Lie algebra (respectively, extended Hamiltonian Lie algebra).

\vspace{1em}

\end{abstract}

\tableofcontents
\newpage

\section{Introduction}
Infinite-dimensional Lie algebras and their representation theory play a crucial role in many areas of Mathematics and Physics. Infinite-dimensional Lie algebras of Cartan type, viz. $ \mathcal{W}_{N}, \mathcal{S}_{N}, \mathcal{H}_{2m}$ and $\mathcal{K}_{2m+1}$ provide an important class of infinite-dimensional Lie algebras. The Virasoro algebra(denoted by $Vir$) is a key infinite-dimensional Lie algebra, which is the universal central extension of rank-1 Witt algebra $\mathcal{W}_{1}$ and plays an important role in theoretical physics, namely in string theory and two-dimensional conformal field theory. One of the important problems is to classify all irreducible Harish-Chandra modules, which are irreducible weight modules with finite weight multiplicity, over various important Lie algebras. A complete classification of irreducible Harish-Chandra modules over the Virasoro Lie algebra was given by Olivier Mathieu in 1992 [13]. The irreducible Harish-Chandra modules for simple finite-dimensional Lie algebras were classified by Olivier Mathieu in [14]. For the classical Lie superalgebras, the Harish-chandra modules were classified by Maria Gorelik and Dimitar Grantcharov(see [19]). 

\par Let $A_{N}=\mathbb{C}[t_{1}^{\pm1},\ldots,t_{N}^{\pm 1}]$ be the algebra of Laurent polynomials in $N$ commuting variables $t_{1},\ldots,t_{N}$ and $\mathcal{W}_{N}=\text{Der}A_{N}$ be the Lie algebra of the derivations of $A_{N}$.  In [15], S. Eswara Rao studied the irreducibility conditions of Shen-Larsson modules for $\mathcal{W}_{N}$ and later showed in [17] that these modules exhaust all the irreducible modules for the extended Witt algebras $\mathcal{W}_{N}\ltimes A_{N}$ with finite dimensional weight spaces. Furthermore, Y. Billig extended this result to the classification of indecomposable modules for the category $\mathcal{J}$ for the Witt algebra $\mathcal{W}_{N}$(see [1]), where the objects in category $\mathcal{J}$ are modules that admit compatible actions of the both the Lie algebra $\mathcal{W}_{N}$ and the commutative algebra $A_{N}$. In [3], Y. Billig and V. Futorny obtained a complete classification of the Harish-Chandra Modules for the Witt algebra $\mathcal{W}_{N}$. J. Talboom studied the Shen-Larsson modules to the subalgebra $\mathcal{S}_{N}$ and they were shown to be irreducible under similar conditions; see [29]. In [2], Y. Billig and J. Talboom classified irreducible and indecomposable modules in the category $\mathcal{J}$ modules for the special algebra $\mathcal{S}_{N}$. In [30], J. Talboom  classified the irreducible and indecomposable modules in the category $\mathcal{J} $-modules for the Hamiltonian algebra $\mathcal{\tilde{H}}_{N}$.

\par Given a Lie algebra $\mathfrak{g}$ over $ \mathbb{C}$ and $B$ a finitely generated commutative associative unital algebra over $\mathbb{C}$, then $\mathcal{M}(\mathfrak{g})=\mathfrak{g}\otimes B$ is a Lie algebra with brackets:
\begin{equation*}
    [x\otimes b_{1}, y\otimes b_{2}]= [x,y]\otimes b_{1}b_{2},
\end{equation*}
where $x,y\in \mathfrak{g} $ and $b_{1}, b_{2} \in B$. The Lie algebra $\mathcal{M}(\mathfrak{g})$ is called the map Lie algebra associated with $\mathfrak{g}$ and $B$. Over the last decade, the study of map Lie algebras and their representation theory has gained the attention of many researchers due to their importance in Mathematics and Mathematical Physics. For example, if $\mathfrak{g}$ is a finite-dimensional simple Lie algebra, and $B=\mathbb{C}[t^{\pm 1}]$, then the universal central extension of $\mathfrak{g}\otimes B$ is the well known affine Kac-Moody algebras. If $\mathfrak{g}$ is a finite-dimensional simple Lie algebra and $B=\mathbb{C}[t]$, then the Lie algebra $\mathfrak{g}\otimes \mathbb{C}[t]$ is an important Lie algebra known as the Current Lie algebra(see [10]). For a finite-dimensional simple Lie algebra $\mathfrak{g}$ and $B=\mathbb{C}[t_{1}^{\pm1},\ldots, t_{N}^{\pm 1}]$($N \geq 2$), the iterated loop algebras $\mathfrak{g}\otimes B$ have the universal central extension known as the Toroidal algebras(see [12]). Let M be a positive integer then $B= \mathbb{C}[t]/(t^{M+1})$, where $(t^{M+1})$ is the principal ideal of $\mathbb{C}[t]$ generated by the element $t^{M+1}$, the associated map algebras $\mathfrak{g}\otimes \mathbb{C}[t]/(t^{M+1})$ are known as the truncated current Lie algebras(see [23]).

\par Irreducible finite-dimensional representations of equivariant map algebras were investigated by E. Neher, A. Savage, and P. Senesi in [26]. Irreducible Harish-Chandra modules for the map and affine Lie superalgebras were classied by L. Calixto, V. Futorny and H. Rocha, as detailed in [6]. Irreducible quasifinite modules over map Virasoro algebras were studied and classified by A. Savage(see [21]). Subsequently, P. Chakraborty and P. Batra constructed a class of irreducible weight modules for these Lie algebras with infinite dimensional weight spaces in [7]. In [20], X. Guo, R. Lu, and K. Zhao gave a complete classification of irreducible Harish-Chandra modules over the loop-Virasoro algebra.   In [9], P. Chakraborty and S. Eswara Rao studied irreducible Harish-Chandra modules of Map extended Witt algebra $ (\mathcal{W}_{N}\ltimes A_{N})\otimes B$ with the assumption that the action of $A_{N}\otimes B$ is associative and the operator $t^{0}\otimes 1$ acts non-trivially and hence by a nonzero scalar(being a central element in $(\mathcal{W}_{N}\ltimes A_{N})\otimes B$). They have shown that any such module is an evaluation module at a single point. This result was later extended to a complete classification of Harish-Chandra modules for the map Witt algebras $\mathcal{W}_{N}\otimes B$ in [27], these modules turn out to be evaluation modules at a single point.

\par  The classification of irreducible integrable modules for loop Toroidal Lie algebras with finite weight multiplicity was established by P. Chakraborty and P. Batra in [5](see also [11]). In [16], S. Eswara Rao studied the irreducible highest weight modules for loop affine Virasoro algebras and  obtained the necessary and sufficient conditions for irreducible highest weight modules to have finite dimensional weight spaces  and S. Eswara Rao, S.S. Sharma and S.  Mukherjee obtained the complete classification for irreducible integrable module of loop affine Virasoro algebra with finite weight multiplicity in [18]. Very recently, in [4](see also [12]), P. Bisht and P. Batra obtained a complete classification of irreducible, integrable modules of map full Toroidal Lie algebras with finite weight multiplicity.
\par The aim of this paper is to study the irreducible uniformly bounded Harish-Chandra modules for the map extended Special Lie algebra and map extended Hamiltonian Lie algebra with the assumption that the action of some nonzero degree element $t^{\textbf{r}}=t^{\textbf{r}}\otimes 1$(for some $\textbf{r}\neq \textbf{0}$) and $t^{\textbf{0}}=t^{\textbf{0}}\otimes 1$ is nontrivial.  The structure of this paper is as follows. \par
In Section 2, we will develop some basic terminology and recall a key result concerning finite-dimensional simple Lie algebras in Subsection 2.1. In Subsection 2.2, the definitions of Witt algebra, Special algebra and Hamiltonian algebra are recalled and some previously obtained results on extended Witt algebras and map extended Witt algebras are stated in Subsection 2.3. 

\par
In the section 3, we present a classification of irreducible uniformly bounded modules for map extended Special algebras $(\mathcal{S}_{N}\ltimes A_{N})\otimes B$ with the assumptions stated above.   In Subsection 3.1, we prove that the action of $A_{N}\otimes B$ on any such module is quasi-associative. Infact, we show that the elements $t^{\textbf{r}}$ acts injectively on module and there exists an algebra homomorphism $\psi: B\to \mathbb{C}$ with $\psi(1)=1$ and a linear map $\phi: B \to \mathbb{C}$ with $\phi(1)=1$ such that $t^{\textbf{r}}\otimes b$ acts as $\psi(b)t^{\textbf{r}}$ and $t^{\textbf{0}}\otimes b$ acts as $ \phi(b)t^{\textbf{0}}$ on $V$ . In Subsection 3.2, we show that $D(u,\textbf{r})b$ acts as $\psi(b)D(u,\textbf{r})$ for every $\textbf{r}\in \mathbb{Z}^{N}\setminus\{\textbf{0}\}$ and $b\in B$, while $D(u,\textbf{0})b$ acts as a  scalar $ f_{\lambda}(u,b)$ on each weight space $V_{\lambda}$ of $V$, where $f_{\lambda}: \mathbb{C}^{N}\times B \to \mathbb{C}$ is a linear map with $f_{\lambda+\sum_{i=1}^{N}r_{i}\delta_{i}}(u,b)=f_{\lambda}(u,b)+(u,\textbf{r})\psi(b)$ for each $\textbf{r}=(r_{i})_{i=1}^{N}\in \mathbb{Z}^{N}$. Finally, we show that $V$ is an irreducible module for the underlying  extended Special algebra $\mathcal{S}_{N}\ltimes A_{N}$  with the action of $A_{N}$ being quasi-associative on $V$. Therefore, $V$ is therefore an irreducible Jet module for the Lie algebra $\mathcal{S}_{N}$. \par 

Section 4 runs parallel to Section 3, we begin with stating a result due to Talboom [30]. Let $V$ be an irreducible uniformly bounded Harish-Chandra module for the Lie algebra $(H_{N}\ltimes A_{N})\otimes B$, with the same assumption as made for $(S_{N}\ltimes A_{N})\otimes B$. In Section 4.1, we show that the action of $A_{N}\otimes B$ is quasi-associative on $V$. The action of $t^{\textbf{r}}b$ in terms of the action of $t^{\textbf{r}}$, similarly as in Subsection 3.1.
The action of $D(u,\textbf{r})b$ $(\textbf{r}\neq \textbf{0})$ is expressed in terms of $D(u,\textbf{r})$ $(\textbf{r}\neq \textbf{0})$, where as $D(u,\textbf{0})b$ acts as scalar $f_{\lambda}(u,\textbf{0})b$ on each weight space $V_{\lambda}$, where $f_{\lambda}:\mathbb{C}^{N}\times B \to \mathbb{C}$ is a bilinear map. Finally, it is shown that $V$ is, in fact, an irreducible module over the underlying extended Hamiltonian algebra $\mathcal{\tilde{H}}_{N}\ltimes A_{N}$, with the quasi-associative action of $A_{N}$. Hence, $V$ is an irreducible Jet module for the Hamiltonian Lie algebra $\tilde{H}_{N}$.
\section{Preliminaries}
\subsection{ Basic Terminology}
Throughout this paper, we denote the set of non negative integers, set of integers, set of positive integers, set of negative integers and the field of complex numbers by $\mathbb{N}$, $ \mathbb{Z}$, $\mathbb{Z}_{+}$, $\mathbb{Z}_{-}$, and $\mathbb{C}$(respectively). All the vector spaces and algebras are over $\mathbb{C}$. Unless otherwise mentioned, the tensor products are considered over $\mathbb{C}$. We recall the following definitions, which will be used in the sequel.
\begin{definition} Let $\mathfrak{g}$ be a Lie algebra with a Cartan subalgebra $\mathfrak{h}$, then a $\mathfrak{g}$-module V is called $\mathfrak{h}$-weight module if it admits a weight space decomposition:
\begin{equation*}
    V = \bigoplus_{\lambda \in \mathfrak{h}^{*}}V_{\lambda},
\end{equation*}
where $V_{\lambda}= \{v\in V \mid h\cdot  v = \lambda(h)v, \; \forall \; h\in \mathfrak{h}\}$. We call $\lambda\in \mathfrak{h}^{*}$ a weight of the module V if $V_{\lambda}\neq 0 $.
\end{definition} 
For a weight module $V$, the set of weights of  $V$ is denoted by $P(V)$.
\begin{definition}
A $\mathfrak{h}$-weight module $V$ for $\mathfrak{g}$ is said to be a uniformly bounded Harish-Chandra Module if there exists a positive integer $M$  such that $\text{dim}V_{\lambda} \leq M$, for all $\lambda \in \mathfrak{h}^{*}$.   
\end{definition}
Let $B$ be a finitely generated commutative associative unital algebra over $\mathbb{C}$.
\begin{definition}
    A module V for $\mathfrak{g}\otimes B$ is called a single point evaluation module if there exists an algebra homomorphism $\psi: B \to C $ such that $x\otimes b\cdot v = \psi(b)x\otimes 1$, where $1$ is the unity of $B$, $x\in \mathfrak{g}$, $ v\in V$ and $b\in B$.
\end{definition} 
Notice that an irreducible single point evaluation $\mathfrak{g}\otimes B$-module is an irreducible $\mathfrak{g}$-module. \\ 
For any $N$-tuples $\textbf{n}=(n_{1},\ldots, n_{N})$, $\textbf{m}=(m_{1},m_{2},\cdots, m_{N})$  of integers and $1\leq i\neq j\leq N$, we denote by $ \text{det}\begin{pmatrix}
    \textbf{m} \\
    \textbf{n}
\end{pmatrix}_{i,j}$ the determinant of the matrix $\begin{bmatrix}
    m_{i} & m_{j} \\
    n_{i} & n_{j}
\end{bmatrix}$.
   \begin{lemma}[\textbf{[17]}]
       \begin{enumerate} \item Let $\mathfrak{g}\subseteq \mathfrak{gl}(V)$ ( $V$ is finite-dimensional) be a nonzero Lie algebra admitting an irreducible representation on $V$. Then $\mathfrak{g}$ is reductive and its center is at most one-dimensional.
       \item Let $\mathfrak{g}$ be a finite-dimensional reductive Lie algebra. Then $\mathfrak{g}=Z(\mathfrak{g})\oplus [\mathfrak{g},\mathfrak{g}]. $\end{enumerate} \end{lemma}
  
   \subsection{ Witt algbera, Special algebra, Hamiltonian algebra}
Let $N$ be a positive integer and  $A_{N}= \mathbb{C}[t_{1}^{\pm{1}}, t_{2}^{\pm{1}},\ldots, t_{N}^{\pm{1}}]$ be the commutative algebra of Laurent polynomials in $N$ variables over $\mathbb{C}$. We now recall the definition of  the  Witt algebras ($\mathcal{W}_{N}$), the Special algebras ($\mathcal{S}_{N}$) and the Hamiltonian algebras $\mathcal{H}_{N}$. For any $N$-tuple of integers $\textbf{m}=(m_{1}, m_{2},..., m_{N})$, we denote the corresponding monomial $t_{1}^{m_{1}}t_{2}^{m_{2}}\cdots t_{N}^{m_{N}}$ by $t^{\textbf{m}}$. If $d_{i}$(for $1\leq i \leq N$) denotes the $i\text{-th}$ degree derivation of $A_{N}$, that is, $ d_{i}\equiv t_{i}\dfrac{d}{dt_{i}}$. Denote by $t^{\textbf{r}}d_{a}$ the derivation $t_{1}^{r_{1}}t_{2}^{r_{2}}\ldots t_{N}^{r_{N}}d_{a}$. For $u\in \mathbb{C}^{N}$ and $\textbf{r}\in \mathbb{Z}^{N}$, let $D(u,\textbf{r})$ denote the derivation $\sum_{i=1}^{N}u_{i}t^{\textbf{r}}d_{i}$. Then we have the following definitions:
\begin{definition}
    The Lie algebra of derivations of $A_{N}$ is called the Witt algebra and is denoted by $\mathcal{W}_{N}$. The bracket structure on $\mathcal{W}_{N}= \text{span}_{\mathbb{C}}\{ D(u,\textbf{r}) \mid u \in \mathbb{C}^{N}, \textbf{r}\in \mathbb{Z}^{N} \}$ is given as follows:
  \begin{equation}
    [D(u,\textbf{r}), D(v,\textbf{s})] = D(w, \textbf{r}+ \textbf{s}),
\end{equation}
where $u,v \in \mathbb{C}^{N}, \textbf{r},\textbf{s} \in \mathbb{Z}^{N}$ and $w= (u|\textbf{s})v-(v|\textbf{r})u$ If $\mathfrak{h}=\text{span}_{\mathbb{C}}\{ d_{i} \mid 1\leq i \leq N\}$, then $\mathfrak{h}$ is a Cartan subalgebra of $\mathcal{W}_{N}$.

 This Lie algebra is naturally identified with the Lie algebra of vector fields on N-dimensional torus. 
\end{definition}
For any $1\leq i\leq N$, we define $\delta_{i}\in \mathfrak{h}^{*}$ by $\delta_{i}(d_{j})=\delta_{ij}$ for every $1\leq j \leq N$, where $\delta_{ij}$ denotes the Kronecker delta. These linear functionals will be used later.

\begin{definition}
  One of the important subalgebras of $\mathcal{W}_{N}$ is the Lie algebra of divergence zero vector fields on N-dimensional torus, known as the Special algebra $\mathcal{S}_{N}$. It is well known that $\mathcal{S}_{N}= \text{span}_{\mathbb{C}}\{D(u,\textbf{r}) \mid u\in \mathbb{C}^{N}, \textbf{r} \in \mathbb{Z}^{N}, (u|\textbf{r})=0\}$. For $ \textbf{r}\in \mathbb{Z}^{N}$, $1\leq i \leq N$ and $1\leq a\neq b\leq N $, let $d_{ab}(\textbf{r})= r_{b}t^{\textbf{r}}d_{a}-r_{a}t^{\textbf{r}}d_{b}$, then the elements $d_{ab}(\textbf{r})$ and $d_{i}$ span $\mathcal{S}_{N}$ as a vector space.
\end{definition}
For $N=2m$($m\in \mathbb{Z}_{+}$) and $\textbf{r}=(r_{1},\cdots,r_{m}, r_{m+1},\cdots r_{2m})\in \mathbb{Z}^{N}$, denote by $\overline{\textbf{r}}$ the element $(r_{m+1},\ldots,r_{2m},-r_{1},\ldots,-r_{m})\in \mathbb{Z}^{N}$.
\begin{definition} Next important subalgebra of $W_{N}$(where $N=2m$ ) is the Lie algebra of Hamiltonian vector fields $\mathcal{H}_{N}$. If $h_{\textbf{r}}=D(\overline{\textbf{r}}, \textbf{r})$, then it is known that $ \mathcal{H}_{N}= \text{span}_{\mathbb{C}}\{h_{\textbf{r}} \mid \textbf{r} \in \mathbb{Z}^{N}  \} $. We further add to $\mathcal{H}_{N}$ the degree derivations i.e. $d_{i}\text{'s}$ to obtain $\tilde{H}_{N} = \mathcal{H}_{N}\rtimes \mathfrak{h}$, then $\mathfrak{h}$ is a Cartan subalgebra of $\tilde {\mathcal{H}}_{N}$.
    
\end{definition}

\subsection{ Map extended Witt algebra and previously obtained results}
As the Lie algebra $\mathcal{W}_{N}$ acts naturally on $A_{N}$, one can define the Lie algebra $\mathcal{W}_{N}\ltimes A_{N}$, we call this Lie algebra the extended Witt algebra. S. Eswara Rao investigated the irreducible modules of $\mathcal{W}_{N}\ltimes A_{N}$ with finite-dimensional weight spaces. He proved that the Larsson-Shen modules exhaust all such modules in terms of the following theorem:
\begin{theorem}{[\textbf{[17]},Theorem 2.9]}
    Let $V$ be an irreducible module of $\mathcal{W}_{N}\ltimes A_{N}$ with finite-dimensional weight spaces with respect to the Cartan subalgebra $\mathfrak{h}$ together with the following assumptions:
    \begin{enumerate}
        \item V admits an $A_{N}$-module structure and the Lie bracket of $A_{N}$ is compatible with the underlying associative algebra structure.
        \item $t^{0}\cdot v = v$ for every $v$ in $V$.
    \end{enumerate}
    Then $V \cong F^{\alpha}(\psi,b)$.
\end{theorem}
Let $B $ be a  finitely generated commutative associative unital algebra over $ \mathbb{C}$. P. Chakraborty and S. Eswara Rao studied the irreducible weight modules for the map extended Witt algebra $ (\mathcal{W}_{N}\ltimes A_{N})\otimes B$ having finite dimensional weight spaces under the assumption that the action of $A_{N}\otimes B$ is associative and $t^{0}\otimes 1$ acts by as the scalar 1 on $V$. The following theorem was proved by them. 
\begin{theorem}{[\textbf{[9]},Theorem 4.5]}
    Let V be an irreducible $ \mathfrak{h}$-weight module for $(\mathcal{W}_{N}\ltimes A_{N})\otimes B$ with finite-dimensional weight spaces. Also, let the action of $A_{N}\otimes B$ on V be associative and $1\cdot v =v$ for all $v$ in $V$. Then there exists an algebra homomorphism $\psi: B \to \mathbb{C}$ such that $x\otimes b\cdot v = \psi(b)x\otimes 1\cdot v$, for all $x\in \mathcal{W}_{N}\ltimes A_{N}$, $b\in B$ and $v\in V$.
\end{theorem}

That is, any such module is, in fact, a single-point evaluation module and therefore an irreducible module for the underlying $\mathcal{W}_{N}\ltimes A_{N}$ with finite-dimensional weight spaces. Later, Sachin Sarma et. al. proved that the action of $A_{N}\otimes B$ is always associative if $V$ is uniformly bounded. In fact, they demonstrated that every irreducible weight module of $\mathcal{W}_{N}\otimes B$ with finite-dimensional weight spaces is a single point evaluation module. 
\subsection{ Map Extended Special algebra and Map extended Hamiltonian algebra  }
Similar to map extended Witt algebra$(\mathcal{W}_{N}\ltimes A_{N})\otimes B$, one can naturally define map extended Special algebra $(\mathcal{S}_{N}\ltimes A_{N})\otimes B $ and map extended Hamiltonian algebra $(\tilde{\mathcal{H}}_{N}\ltimes A_{N})\otimes B $. Note that $ \mathfrak{h}\equiv \mathfrak{h}\otimes 1 $ plays the role of Cartan subalgebra for each of these algebras, where 1 is the unity of $B$. Now we define the associative action and quasi-associative action of $A_{N}\otimes B$ for a given irreducible $(\mathcal{S}_{N}\ltimes A_{N})\otimes B $-module. We assume that $t^{\textbf{0}}\otimes1$ acts as a scalar $c\neq 0$. The definitions are similar for $(\mathcal{\tilde{H}}\ltimes A_{N})\otimes B$.
\begin{definition}
    The action of $A_{N}\otimes B$ is said to be associative, if we have \begin{equation*}
        t^{\textbf{m}}b_{1}t^{\textbf{n}}b_{2}= t^{\textbf{m}+\textbf{n}}b_{1}b_{2}, \; \text{on} \; V \: \forall \textbf{m},\textbf{n} \in \mathbb{Z}^{N} \; \text{and} \; b_{1},b_{2}\in B.
    \end{equation*}
\end{definition}
\begin{definition}
    The action of $A_{N}\otimes B$ is said to be quasi-associative, if there exists an algebra homomorphism $\psi: B \to \mathbb{C}$ with $\psi(1)=1$, a linear map $\phi: B \to \mathbb{C}$ with $\phi(1)=1$ and nonzero complex numbers $\lambda$ and $\mu$ such that
    \begin{align*}
        t^{\textbf{m}}\otimes b_{1}\cdot t^{\textbf{n}}\otimes b_{2} &= \lambda t^{\textbf{m}+\textbf{n}}\otimes b_{1}b_{2}=\lambda \psi(b_{1}b_{2})t^{\textbf{m}+\textbf{n}}\otimes1, \: \text{on} \; V \\ &\;  \; \; \; \; \; \; \; \; \; \;  \text{for} \; \textbf{m},\textbf{n},\textbf{m}+\textbf{n}\in \mathbb{Z}^{N}\setminus\{\textbf{0}\}, b_{1},b_{2}\in B,
       \\
 t^{\textbf{m}}\otimes b_{1}\cdot t^{\textbf{0}}\otimes b_{2} &= c\phi(b_{2})t^{\textbf{m}}\otimes b_{1}, \; \text{on}\; V \; \text{for} \; \textbf{m}\in \mathbb{Z}^{N}, b_{1},b_{2}\in B, 
 \\
 t^{\textbf{m}}\otimes b_{1}\cdot t^{-\textbf{m}}\otimes b_{2}&
 =\mu\psi(b_{1}b_{2})t^{\textbf{0}}\otimes 1, \; \text{on} \;  V \;  \text{for} \; \textbf{m}\in \mathbb{Z}^{N}\setminus\{\textbf{0}\}, b_{1},b_{2}\in B,
    \end{align*}
where $ \lambda^{2}=c\mu$.
\end{definition}
\section{ Harish Chandra Modules for map extended Special algebra }
We first recall some of the known results for the underlying "Extended Special algebra". Y. Billig and J. Talboom classified the irreducible and indecomposable Jet modules for the Lie algebra of divergence zero vector fields on $N$-dimensional torus. We recall their result for the irreducible jet modules below:
\begin{theorem}{[\textbf{[2]},Theorem 5.2]}
    Suppose $W$ be a finite-dimensional irreducible module for $\mathfrak{sl}_{N}$, we extend it a $\mathfrak{gl}_{N}$-module by assuming that identity matrix acts trivially on W. Let $\alpha, \beta \in \mathbb{C}^{N}$.  Consider $W\otimes A_{N}$ and the following action of $\mathcal{S}_{N}\ltimes A_{N}$ on it:
    \begin{enumerate}
        \item $D(u,\textbf{r})\cdot w\otimes t^{\textbf{s}}=(u|\textbf{s}+\beta)w\otimes t^{\textbf{s}+\textbf{r}} + \sum_{i,j}u_{i}r_{j}E_{ji}w\otimes t^{\textbf{s}+\textbf{r}}$, for $\textbf{r}\neq 0$.
        \item $D(u,\textbf{0})\cdot w\otimes t^{\textbf{s}} = (u, \alpha + \textbf{s})w\otimes t^{\textbf{s}}$.
        \item $t^{\textbf{r}}\cdot w\otimes t^{\textbf{s}}= w\otimes t^{\textbf{s}+\textbf{r}}$.
    \end{enumerate}
    Every irreducible Jet module for $\mathcal{S}_{N}$ with finite-dimensional weight spaces with respect to $\mathfrak{h}$ can be expressed in the form described above.
\end{theorem}
We now investigate the irreducible uniformly bounded weight modules of $(\mathcal{S}_{N}\ltimes A_{N})\otimes B$, under the condition that some $t^{\textbf{r}}\otimes 1$(for $\textbf{r}\neq \textbf{0}$) and $t^{\textbf{0}}\otimes 1$ act nontrivially. Let $V$ be a uniformly bounded weight module for $(\mathcal{S}_{N}\ltimes A_{N})\otimes B$ under the assumptions above. Let $A_{N}' = \bigoplus_{\textbf{0}\neq \textbf{r}\in \mathbb{Z}^{N}}\mathbb{C}t^{\textbf{r}}$, then $(\mathcal{S}_{N}\ltimes A_{N})\otimes B= ((\mathcal{S}_{N}\ltimes A_{N}')\otimes B) \oplus (t^{0}\otimes B)$. As $t^{0}\otimes B$ being central in $(\mathcal{S}_{N}\ltimes A_{N})\otimes B$ and $V$ is irreducible uniformly bounded weight module, so there exists a linear map $\phi: B\to \mathbb{C}$ with $\phi(1)=1$ such that $t^{0}\otimes b\cdot v = c\phi(b)v$ for all $v\in V$ and $b\in B$, where $c$ is the scalar by which $t^{\textbf{0}}\otimes 1$ acts on $V$. We see that the action of $A_{N}\otimes B$ on $V$ is quasi-associative.
\subsection{ Quasi-associative action of $A_{N}\otimes B$}
We have the following lemma, the proof of which is similar to  that in [20].
\begin{lemma}
    Let $V=\bigoplus_{\textbf{n}\in \mathbb{Z}^{N}}V_{\textbf{n}}$ be an irreducible uniformly bounded
    $\mathbb{Z}^{N}$-graded module  over the associative algebra $A_{N}$, then $\text{dim}V_{\textbf{n}}\leq 1$, for all $n\in \mathbb{Z}^{N}$.
\end{lemma}
\begin{lemma}
    Suppose that $V$ is an irreducible module for $(\mathcal{S}_{N}\ltimes A_{N})\otimes B $ and $x\in \mathcal{U}(A_{N}\otimes B)$. If $x\cdot v = 0$ for some nonzero $v\in V$, then x is locally nilpotent on $V$.
\end{lemma}
\begin{proof}
    Note that $[[(\mathcal{S}_{N}\ltimes A_{N})\otimes B, A_{N}\otimes B], A_{N}\otimes B]=0$, therefore $[[D(u,r)\otimes b,x],x]\cdot v=0$, which gives $x^{2}\cdot D(u,r)\otimes b\cdot v =0$. It follows by induction that $ x^{k+1}\cdot D(u_{1},\textbf{r}_{1})\otimes b_{1}\cdots D(u_{k},\textbf{r}_{k})\otimes b_{k}\cdot v =0 $, for every $k\in \mathbb{Z}_{+}$, $u_{i}\in \mathbb{C}^{N}$, $\textbf{r}_{i}\in \mathbb{Z}^{n}$ with $(u_{i}|\textbf{r}_{i})=0$($\forall 1 \leq i \leq k$). The result then follows from the irreducibility of $V$.
\end{proof}
Let $V=\bigoplus_{\lambda \in \mathfrak{h}^{*}} V_{\lambda}$ be an irreducible uniformly bounded Harish-Chandra module for $(\mathcal{S}_{N}\ltimes A_{N})\otimes B$. Fix a weight $\lambda $ of $V$. By the irreducibility of $V$, it follows that $P(V)\subseteq \{\lambda + \sum_{i=1}^{N}r_{i}\delta_{i} \mid r_{i}\in \mathbb{Z} \}$. Let $\Lambda =(\lambda(d_{i}))_{i=1}^{N} $, then it is convenient to use the notation $V=\bigoplus_{\textbf{r}\in \mathbb{Z}^{N}} V_{\Lambda +\textbf{r}}$, where $V_{\Lambda + \textbf{r}}= V_{\lambda + \sum_{i=1}^{N}r_{i}\delta_{i}}$, for every $\textbf{r}=(r_{i})_{i=1}^{N}\in \mathbb{Z}^{N}$. \\
We prove the following proposition, the proof of which is a generalization of Proposition 3.4 of [20]. 
\begin{proposition}
    Let $V$ be a irreducible uniformly bounded Harish-Chandra module for $(\mathcal{S}_{N}\ltimes A_{N})\otimes B$ with respect to $\mathfrak{h}$. Then  for a given $ b\in B$ either all $t^{\textbf{m}}\otimes b$, $\textbf{m}\neq \textbf{0}$ act injectively on V or all $ t^{\textbf{m}}\otimes b$, $\textbf{m}\neq \textbf{0}$ act nilpotently on V.
\end{proposition}

\begin{proof}
Suppose $t^{\textbf{m}}b$ acts injectively on $V$. For the sake of simplicity, we use the notation $t^{\textbf{m}}b= t^{\textbf{m}}\otimes b$. We proceed with proving the following claims.
\\ 
\textbf{Claim 1:} For any $\textbf{n}\in \mathbb{Z}^{N}$ with $\det
\begin{pmatrix}
\textbf{n} \\
\textbf{m}
\end{pmatrix}_{i,j}\neq 0
 $,  $1\leq i\neq j \leq N$, if $t^{\textbf{n}}b$ is locally nilpotent on $V$, then it is also nilpotent on $V$.
 \\
 Let $\textbf{n}\in \mathbb{Z}^{N}$ be such that $\det
\begin{pmatrix}
\textbf{n} \\
\textbf{m}
\end{pmatrix}_{i,j}\neq 0
 $, for some  $1\leq i\neq j \leq N$ and $t^{\textbf{n}}b$ is locally nilpotent on $V$. Since $t^{-\textbf{n}}bt^{\textbf{n}}b$ is locally nilpotent on each graded component and  $V$ is uniformly bounded, there exists a positive integer say $N'$ such that $ (t^{-\textbf{n}}b\cdot t^{\textbf{n}}b)^{N'}\cdot V =0 $, i.e. $(t^{-\textbf{n}}b)^{N'}\cdot(t^{\textbf{n}}b)^{N'}\cdot V = 0$. By applying the operator $d_{ij}(\textbf{m+n}) = (m_{j}+n_{j})t^{\textbf{m}+\textbf{n}}d_{i}- (m_{i}+n_{i})t^{\textbf{m}+\textbf{n}
 }d_{j}$ to the above expression, we obtain 
 \begin{align*}
0 &= d_{ij}(\textbf{m}+\textbf{n})\cdot (t^{-\textbf{n}}b)^{N'}\cdot (t^{\textbf{n}}b)^{N'}\cdot V \\
 &= N' \text{det} \begin{pmatrix}
     \textbf{n} \\
     \textbf{m}
 \end{pmatrix}_{i,j} t^{\textbf{m}+ 2\textbf{n}}b\cdot (t^{-\textbf{n}}b)^{N'}\cdot (t^{\textbf{n}}b)^{N'-1}\cdot V - N'\text{det}\begin{pmatrix}
     \textbf{n} \\ \textbf{m}
 \end{pmatrix}_{i,j}t^{\textbf{m}}b(t^{-\textbf{n}}b)^{N'-1}(t^{\textbf{n}}b)^{N'}\cdot V.
\\
\end{align*}
 Operating  $t^{\textbf{n}}b$ in the above equation, we obtain $ t^{\textbf{m}}b(t^{-\textbf{n}}b)^{N'-1}(t^{\textbf{n}}b)^{N'+1}\cdot V =0$, and hence $(t^{-\textbf{n}}b)^{N'-1}(t^{\textbf{n}}b)^{N'+1}\cdot V =0 $, since $t^{\textbf{m}}b$ is assumed to be injective. Then it follows by  descending induction that $ (t^{\textbf{n}}(b))^{N_{0}}\cdot V=0$.
\\
\textbf{Claim 2:} $t^{q\textbf{m}/p}b$ acts injectively on $V$ for any $p, q \in \mathbb{Z}$, whenever $q\textbf{m}/p \in \mathbb{Z}^{N}\ \setminus \{ 0\}$. \\
Without loss of generality, we can assume that $q/p <0$. Assume on the contrary that $t^{q\textbf{m}/p}b$ is not injective on $V$. Lemma 3.3 gives that $t^{q\textbf{m}/p}b$ is locally nilpotent on $V$. Note that the operator $(t^{\textbf{m}}b)^{-q}(t^{q\textbf{m}/p}b)^{p}$ is locally nilpotent on each graded component $ V_{\Lambda +\textbf{m}}$, for every $\textbf{m}\in \mathbb{Z}^{N}$ and $V$ is uniformly bounded there we can find $N''\in \mathbb{Z}_{+}$, such that $((t^{\textbf{m}}b)^{-q}(t^{q\textbf{m}/p}b)^{p})^{N''}\cdot V= (t^{\textbf{m}}b)^{-qN''}(t^{q\textbf{m}/p}b)^{pN''}\cdot V = 0$, as $t^{\textbf{m}}b$ acts injectively on $V$, it follows that $(t^{q\textbf{m}/p}b)^{pN''}=0$. Let $N_{0}''\in \mathbb{Z}_{+}$ be the minimal such that $(t^{q\textbf{m}/p}b)^{N_{0}''}\cdot V=0$. Let $\textbf{n}\in \mathbb{Z}^{N}$ be such that $\text{det} \begin{pmatrix}
    \textbf{n} \\
    \textbf{m}
\end{pmatrix}_{i,j}\neq 0$ for some $1\leq i \neq j \leq N$, then we have \begin{equation*}
    d_{ij}(\textbf{m}-\textbf{n})(t^{q\textbf{m}/p}b)^{N_{0}''}\cdot V = N_{0}''\text{det} \begin{pmatrix}
        q\textbf{m}/p \\
        \textbf{n}
    \end{pmatrix}_{i,j}t^{\textbf{n}}b\cdot (t^{q\textbf{m}/p}b)^{N_{0}''-1}\cdot v.
\end{equation*}
By the minimality of $N_{0}''$, we get $0\neq v\in V$ such that $t^{\textbf{n}}b\cdot v=0$, therefore, by claim 1, $t^{\textbf{n}}b$ is nilpotent on $V$. Let $N_{0}\in \mathbb{Z}_{+}$ be the minimal such that $(t^{\textbf{n}}b)^{N_{0}}\cdot V=0$.  Consider \begin{equation*}
 0=   d_{ij}(\textbf{m}-\textbf{n})\cdot (t^{\textbf{n}}b)^{N_{0}}\cdot V = N_{0}\text{det}\begin{pmatrix}
        \textbf{n} \\
        \textbf{m}
    \end{pmatrix}_{i,j} t^{\textbf{m}}b \cdot (t^{\textbf{n}}b)^{N_{0}-1}\cdot v.
\end{equation*}
\\
By the minimality of $N_{0}$, we obtain that $t^{\textbf{m}}b$ is not injective, a contradiction. The claim therefore follows. 
\\
\textbf{Claim 3:} $ t^{\textbf{n}}b$ acts injectively on $V$ for all $\textbf{0}\neq \textbf{n}\in \mathbb{Z}^{N}$.  
\\ We first show that, using claim 2, it suffices to prove that $t^{\textbf{n}}b$ is injective for any $\textbf{n}\in \mathbb{Z}^{N}$ with $\text{det}\begin{pmatrix}
    \textbf{n} \\ \textbf{m}
\end{pmatrix}_{i,j}\neq 0$, for some $1\leq i\neq j\leq N$. If $ \text{det}\begin{pmatrix}
    \textbf{n} \\ \textbf{m}
\end{pmatrix}_{i,j}= 0$, for all $1\leq i\neq  j \leq N$, then one can verify that $\textbf{n}$ is a rational multiple of $\textbf{m}$ and hence in this case the claim follows from claim 2  . Since $\textbf{m}\neq 0 $, without loss of generality we can assume that $m_{1}\neq 0$ and  consider $\text{det}\begin{pmatrix}
    \textbf{n}\\\textbf{m}
\end{pmatrix}_{i,j} = 0$, $\forall \;  1\leq i \neq j \leq N$, which gives that $r_{j}(m_{1}, m_{j}) = (n_{1},n_{j})$, for some rational numbers $r_{j} $, $2\leq j \leq N$. Since $\textbf{n}\neq 0$, all $r_{i}$'s cannot be simultaneously zero, therefore $r_{i}=r\neq 0 $ for all $i\in \{1,\ldots,N\}$, and therefore it follows that $\textbf{n} = r\textbf{m}$ and hence in this case injectivity follows from claim 2.   
\par Let $\textbf{n}\in \mathbb{Z}^{N}$ with $ \text{det}\begin{pmatrix}
    \textbf{n} \\
    \textbf{m}
\end{pmatrix}_{i,j}\neq 0$, for some $i\neq j \in \{1,2,\ldots, N\}$ be such that $t^{\textbf{n}}b$ is locally nilpotent on $V$, then by Claim 1 there exists $k \in \mathbb{Z}_{+}$ such that $(t^{\textbf{n}}b)^{k}\cdot V =0$. 
 Since $V$ is uniformly bounded, we can assume that $\text{dim}V_{\Lambda}\geq \text{dim}V_{\Lambda +\textbf{r}}$, for every $\textbf{r}\in \mathbb{Z}^{N}$. Then we have $t^{\textbf{m}}bV_{\Lambda +i\textbf{m}}=V_{\Lambda + (i+1)\textbf{m}}$ and $\text{dim}V_{ \Lambda +i\textbf{m}}=\text{dim}V_{\Lambda}$, for every $i\in \mathbb{Z}$. Then for any $v\in V_{\Lambda-k\textbf{m}}$, we show that 
 \begin{equation}
(d_{ij}(\textbf{m}-\textbf{n}))^{k}(t^{\textbf{n}}b)^{k}\cdot v = k!\text{det}\begin{pmatrix}
    \textbf{n} \\
    \textbf{m}
\end{pmatrix}_{i,j}^{k}(t^{\textbf{m}}b)^{k}\cdot v +  t^{\textbf{n}}b\cdot u,
 \end{equation}
 for some $u\in V_{\Lambda-\textbf{n}}$. For $k=1$, we have \begin{equation*}
     d_{ij}(\textbf{m}-\textbf{n})t^{\textbf{n}}b\cdot v = \text{det}\begin{pmatrix}
         \textbf{n} \\
         \textbf{m}
     \end{pmatrix}_{i,j}t^{\textbf{m}}b\cdot v + t^{n}b\cdot d_{ij}({\textbf{m}-\textbf{n}})\cdot v.
 \end{equation*}
 Let the above expression be true for $k$, then for $v\in V_{\Lambda - (k+1)\textbf{m}}$, we have \begin{align}
     d_{ij}(\textbf{m}-\textbf{n})^{k+1}(t^{\textbf{n}}b)^{k+1}\cdot v &= d_{ij}(\textbf{m}-\textbf{n})^{k}d_{ij}(\textbf{m}-\textbf{n})(t^{\textbf{n}}b)^{k+1}\cdot v \\
     &=d_{ij}({\textbf{m}-\textbf{n}})^{k}((k+1)\text{det}\begin{pmatrix}
         \textbf{n} \\
         \textbf{m}
     \end{pmatrix}_{i,j} (t^{\textbf{n}}b)^{k}t^{\textbf{m}}b + (t^{\textbf{n}}b)^{k+1}d_{ij}(\textbf{m}-\textbf{n}))\cdot v.
 \end{align}
 We note that in the last expression $t^{\textbf{m}}b\cdot v \in V_{\Lambda -k\textbf{m}}$ and in $t^{\textbf{n}}b\cdot d_{ij}(\textbf{m}-\textbf{n})\cdot v \in V_{\Lambda -k\textbf{m}}$, therefore by induction we have
 \begin{equation}
     d_{ij}(\textbf{m}-\textbf{n})^{k}(t^{\textbf{n}}b)^{k}t^{\textbf{m}}b\cdot v = k!\text{det}\begin{pmatrix}
         \textbf{n} \\ \textbf{m}
     \end{pmatrix}_{i,j}^{k}(t^{\textbf{m}}b)^{k}\cdot t^{m}b\cdot v + t^{\textbf{n}}bu_{1},
 \end{equation}
 for some $u_{1}\in V_{\Lambda-\textbf{n}}$ and 
 \begin{equation}
     d_{ij}(\textbf{m}-\textbf{n})^{k}(t^{\textbf{n}}b)^k\cdot t^{\textbf{n}}b \cdot d_{ij}(\textbf{m}-\textbf{n})\cdot v= k1\text{det}\begin{pmatrix}
         \textbf{n} \\
         \textbf{m}
     \end{pmatrix}_{i,j}^{k}(t^{\textbf{m}}b)^{k}\cdot t^{\textbf{n}}b\cdot d_{ij}(\textbf{m}-\textbf{n})\cdot v + t^{\textbf{n}}u_{2},
 \end{equation}
 where $u_{2}\in V_{\Lambda-\textbf{n}}$, and $(t^{\textbf{m}}b)^{k}\cdot d_{ij}(\textbf{m}-\textbf{n})\cdot v\in V_{\Lambda-\textbf{n}}$, therefore, RHS of the expression (3.4.6) is of the form $t^{\textbf{n}}b\cdot u'$, for some $u'\in V_{\Lambda-\textbf{n}}$. Combining (3.4.4),(3.4.5) and (3,4,6), we obtain that
 \begin{equation}
     d_{ij}(\textbf{m}-\textbf{n})^{k+1}(t^{\textbf{n}}b)^{k+1}\cdot v = (k+1)!\text{det}\begin{pmatrix}
         \textbf{n} \\
         \textbf{m}
     \end{pmatrix}_{i,j}(t^{\textbf{n}}b)^{k+1}\cdot v + t^{n}(b)\cdot u.
 \end{equation}
 As have choosen $k$ such that $(t^{n}b)^{k}\cdot V=0$, therefore \begin{equation*}
  0= d_{ij}(\textbf{m}-\textbf{n})^{k}\cdot (t^{\textbf{n}}b)^{k}\cdot v =  k!\text{det}\begin{pmatrix}
         \textbf{n} \\
         \textbf{m}
     \end{pmatrix}_{i,j}^{k}(t^{\textbf{m}}b)^{k}\cdot v + t^{\textbf{n}}b\cdot u,
 \end{equation*}
 for some $u\in V_{\Lambda -\textbf{n}}$, therefore $t^{\textbf{n}}b: V_{\Lambda-\textbf{n}}\to V_{\Lambda}$ is surjective, and the fact $\text{dim}V_{\Lambda }\geq \text{dim}V_{\Lambda -\textbf{r}} $ gives that $t^{\textbf{n}}b: V_{\Lambda -\textbf{n}}\to V_{\Lambda }$ is bijective. Repeating this argument, we see that $t^{\textbf{n}}b: V_{\Lambda -(i+1)\textbf{n}}\to V_{\Lambda -i\textbf{n}}$ is bijective for every $i\in \mathbb{N}$. Therefore, $t^{\textbf{n}}b$ cannot be nilpotent on $V$, a contradiction.
\end{proof}
\begin{theorem}
    If $t^{\textbf{r}}b $ acts locally nilpotently for some $\textbf{n}\neq 0$, then $t^{\textbf{s}}\otimes bB\cdot V=0$ for every $\textbf{s}\in \mathbb{Z}^{N}\setminus\{\textbf{0}\}$ and hence $A'\otimes bB\cdot V=0$, where $bB$ is the principal ideal $<b>$ of $B$ generated by $b$ . 
\end{theorem}
\begin{proof} As $t^{\textbf{r}}b$ is locally nilpotent on $V$, by Proposition 3.4, it follows that $t^{\textbf{n}}b$ is locally nilpotent on $V$ $\forall \:  \textbf{n} \in \mathbb{Z}^{N}\setminus \{\textbf{0}\}$. As $V$ is uniformly bounded weight module, there exists $N'\in \mathbb{Z}_{+}$ such that $(t^{-\textbf{n}}b\cdot t^{\textbf{n}}b)^{N'}\cdot V=0$, for every $\textbf{n}\in \mathbb{Z}^{N}\setminus \{\textbf{0}\}$.\\ 
For any $\textbf{m}\neq 0$ with $ \text{det} \begin{pmatrix}
    \textbf{n} \\ \textbf{m}
\end{pmatrix}_{i,j}\neq 0 $, for some $1\leq i\neq j \leq N$, and $v\in V$, we have
\begin{multline}
  0 = d_{ij}(\textbf{m})b_{1}(t^{-\textbf{n}}b\cdot t^{\textbf{n}}b)^{N'}= \text{det}\begin{pmatrix}
      \textbf{n}\\ \textbf{m}
  \end{pmatrix}_{i,j}N't^{\textbf{n}+\textbf{m}}bb_{1} \cdot (t^{\textbf{n}}b)^{N'-1}\cdot (t^{-\textbf{n}}b)^{N'}\cdot v
  \\ 
   - \text{det}\begin{pmatrix}
      \textbf{n}\\ \textbf{m}
  \end{pmatrix}_{i,j}N't^{\textbf{m}-\textbf{n}}bb_{1}\cdot(t^{-\textbf{n}}b)^{N'-1}\cdot (t^{\textbf{n}}b)^{N'}\cdot v.
\end{multline}
    Applying $t^{\textbf{n}}b$ to (3.5.8), we get 
    \begin{equation}
        t^{\textbf{m}-\textbf{n}}bb_{1}\cdot(t^{-\textbf{n}}b)^{N'-1}\cdot (t^{\textbf{n}}b)^{N'+1 }\cdot V =0, 
    \end{equation}
    $\forall \: \textbf{m}\in \mathbb{Z}^{N}\setminus\{\textbf{0}\} \: \text{with} \: \text{det}\begin{pmatrix}
            \textbf{n} \\ \textbf{m}
        \end{pmatrix}_{i,j}\neq 0, \: \text{for some} \:  1\leq i\neq j \leq N$. For $1\leq j \leq N$, we show that 
        \begin{equation}
            t^{\textbf{m}_{1}-\textbf{n}}bb_{1}\cdots t^{\textbf{m}_{j}-\textbf{n}}bb_{j}(t^{-\textbf{n}}b)^{N-j}\cdot (t^{\textbf{n}}b)^{N+j}\cdot V = 0, 
        \end{equation}
         for $\textbf{m}_{1},...,\textbf{m}_{j}\in \mathbb{Z}^{N}\setminus\{\textbf{0}\}$ with $\sum_{i=1}^{j}\epsilon_{i}\textbf{m}_{i}\notin \mathbb{Q}\textbf{n}$, for all $\epsilon_{i} \in \{0,1\}$. For $j=1$, we have already done (3.5.9). Suppose that (3.5.10) is true for some $j$ such that $1\leq j < N$. Then for any $\textbf{m}_{j+1}\in \mathbb{Z}^{N}\setminus\{\textbf{0}\}$ with $\textbf{m}_{j+1}+\sum_{i=1}^{j}\epsilon_{i}\textbf{m}_{i} \notin \mathbb{Q}\textbf{n}$, we have 
          \begin{align*}
              d_{ij}(\textbf{m}_{j+1})b_{j+1}\cdot t^{\textbf{m}_{1}-\textbf{n}}bb_{1}\cdots t^{\textbf{m}_{j}-\textbf{n}}bb_{j}(t^{-\textbf{n}}b)^{N-j}(t^{\textbf{n}}b)^{N+j}\cdot V=0, 
              \end{align*}
      which gives \begin{align*}
          -(N-j)\text{det}\begin{pmatrix}
      \textbf{n}\\ \textbf{m}
  \end{pmatrix}_{i,j}t^{\textbf{m}_{j+1}-\textbf{n}}bb_{j+1}t^{\textbf{m}_{1}-\textbf{n}}bb_{1}\cdots t^{\textbf{m}_{j}-\textbf{n}}bb_{j}(t^{-\textbf{n}}b)^{N'-j-1}(t^{\textbf{n}}b)^{N'+j}\cdot V \\ + (N+j)\text{det} \begin{pmatrix}
      \textbf{n}\\ \textbf{m}
  \end{pmatrix}_{i,j}t^{\textbf{m}_{j+1}-\textbf{n}}bb_{j+1}t^{\textbf{m}_{1}-\textbf{n}}bb_{1}\cdots t^{\textbf{m}_{j}-\textbf{n}}bb_{j}(t^{-\textbf{n}}b)^{N'-j}(t^{\textbf{n}}b)^{N'+j-1}=0,
      \end{align*}
      \\
      Operating $t^{\textbf{n}}b$ to equation above and using (3.5.10), it follows that 
      \begin{equation*}
          t^{\textbf{m}_{1}-\textbf{n}}bb_{1}\cdots t^{\textbf{m}_{j+1}-\textbf{n}}bb_{j+1}(t^{-\textbf{n}}b)^{N-(j+1)}(t^{\textbf{n}}b)^{N+j+1}\cdot V=0,
      \end{equation*}
      for every $\textbf{m}_{1}, \textbf{m}_{2},\ldots, \textbf{m}_{j+1} \in \mathbb{Z}^{N} $ with $\sum_{i=1}^{j+1}\epsilon_{i}\textbf{m}_{i} \notin \mathbb{Q}\textbf{n}$. Therefore, equation (3.5.10) follows by induction.  \\
      Setting $j=N$ in (3.5.10), we obtain: $ t^{\textbf{m}_{1}-\textbf{n}}bb_{1}\cdots t^{\textbf{m}_{N}-\textbf{n}}bb_{N}\cdot (t^{\textbf{n}}b)^{2N}\cdot V=0$. \par
        Similarly, we can prove for $1\leq j \leq 2N$ that 
        \begin{multline}
            t^{\textbf{m}_{1}-\textbf{n}}bb_{1}\cdots t^{\textbf{m}_{N+j}-\textbf{n}}bb_{N+j}\cdot (t^{\textbf{n}}b)^{2N-j}\cdot V=0, \: \\  \forall \: \textbf{m}_{1}, \textbf{m}_{2}, \ldots, \textbf{m}_{N+j} \in \mathbb{Z}^{N}, \sum_{i=1}^{N+j} \epsilon_{i}\textbf{m}_{i} \notin \mathbb{Q}\textbf{n}.
            \end{multline}
        Substituting $j=2N$ into equation (3.5.11), we obtain: 
        \begin{equation}
            t^{\textbf{m}_{1}-\textbf{n}}bb_{1}\cdots t^{\textbf{m}_{3N}-\textbf{n}}bb_{3N}\cdot V =0, \forall \:  \sum_{i=1}^{3N} \epsilon_{i}\textbf{m}_{i} \notin \mathbb{Q}\textbf{n}.
        \end{equation}
        For any $\textbf{n}_{1},\ldots, \textbf{n}_{3N}\in \mathbb{Z}^{N}$, we choose $\textbf{n}\in \mathbb{Z}^{N}\setminus \{ \sum_{i=1}^{3N} \epsilon_{i} \textbf{n}_{i}, \epsilon_{i}=0 \:\text{or} \: 1 \}$ and let $\textbf{m}_{i}=\textbf{n}+\textbf{n}_{i}$, $ \forall \: 1\leq i \leq 3N$, we get that $t^{\textbf{n}_{1}}bb_{1}\cdots t^{\textbf{n}_{3N}}bb_{3N}\cdot V=0 $, and hence $(A'\otimes bB)^{3N}\cdot V =0$. Observe that $V'= \{v \in V \mid A'\otimes bB \cdot V =0 \}$ is a non-zero $(\mathcal{S}_{N}\ltimes A_{N})\otimes B$-submodule of $ V$ and, by the irreducibility of $V$, we have $V'=V$.
\end{proof}
As the action of $A_{N}'= A_{N}'\otimes 1$ is nontrivial, Proposition 3.4 and Theorem 3.5 imply that $ t^{\textbf{r}} $ is injective for all $\textbf{0}\neq\textbf{r}\in \mathbb{Z}^{N}$. Consequently, we have $\text{dim}(V_{\Lambda +\textbf{m}})=\text{dim}(V_{\Lambda +\textbf{n}})=m, \: \forall \: \textbf{m},\textbf{n}\in \mathbb{Z}^{N}$. Let $\{ v_{1}, v_{2},\ldots, v_{m}\}$ be a basis for $V_{\Lambda }$ and $t^{\textbf{r}}\cdot v_{i} = v_{i}(\textbf{r})(\: \forall \: 1\leq i \leq m, \: \: \textbf{r}\neq \textbf{0} )$, therefore $\{v_{i}(\textbf{r})\}_{i=1}^{m}$ is a basis for $ V_{\Lambda +\textbf{r}}$, $\textbf{r} \neq \textbf{0}$. Now for $b\in B$, the $t^{\textbf{r}}b\cdot(v_{i})_{i=1}^{m}= (v_{i}(\textbf{r}))\textbf{B}_{\textbf{r}, b}$, if $\lambda_{\textbf{r}, b}$ is an eigenvalue of $B_{\textbf{r}, b}$, then $(t^{\textbf{r}}b-\lambda_{\textbf{r},b}t^{\textbf{r}})\cdot v =0$, for some $0\neq v \in V_{\Lambda}$, Lemma 3.3 implies that $t^{\textbf{r}}b-\lambda_{\textbf{r},b}t^{\textbf{r}}$ is locally nilpotent on $V$. For any $k\in \mathbb{N}$, $(t^{\textbf{r}}b -\lambda_{\textbf{r},b}t^{\textbf{r}})^{k}(v_{i})_{i=1}^{m}=(t^{\textbf{r}})^{k}(v_{i})_{i=1}^{m}(\textbf{B}_{\textbf{r},b}-\lambda_{\textbf{r},b}I)^{k}  $ and since $t^{\textbf{r}}$ is injective, it follows that $ \textbf{B}_{\textbf{r},b}-\lambda_{\textbf{r},b}\textbf{I}$  is a nilpotent matrix. As a consequence, all the eigenvalues of $t^{\textbf{r}}b$ are the same. It is easy to see that $\lambda_{\textbf{r},1}=1$. The following lemma shows that for any $\textbf{r}\neq 0$, the assignment $b \mapsto \lambda_{\textbf{r},b}$ defines a linear map from $B$ to $\mathbb{C}$.
\begin{lemma}
    For any $\textbf{r}\neq 0$, the map $b\to \lambda_{\textbf{r},b}$ is a linear map from $B$ to $\mathbb{C}$.
\end{lemma}
\begin{proof}
    Consider the Lie algebra generated by the operators $t^{\textbf{r}}b, b\in B $, these operators determine a family of matrices $\textbf{B}_{\textbf{r},b}$, which commute with each other and hence form an abelian subalgebra of $\mathfrak{gl}_{m}$, therefore there exists a common eigenvector $v$ of the family of matrices $\textbf{B}_{\textbf{r}, b}$. In other words, there is a nonzero vector $v\in V_{\Lambda}$ such that $t^{\textbf{r}}b\cdot v = \lambda_{\textbf{r}, b}t^{\textbf{r}}\cdot v, \: \forall \: b\in B$. \par
   Let $v$ be a common eigen vector of the family of matrices $ \textbf{B}_{\textbf{r},b}$,  and consider $\lambda_{\textbf{r},b_{1}+b_{2}}t^{\textbf{r}}\cdot v = t^{\textbf{r}}(b_{1}+b_{2})\cdot v = t^{\textbf{r}}b_{1}\cdot v + t^{\textbf{r}}b_{2}\cdot v = (\lambda_{\textbf{r},b_{1}}+\lambda_{\textbf{r},b_{2}})t^{\textbf{r}}\cdot v$. \\
    Since $t^{\textbf{r}}$ is injective, it follows that $\lambda_{\textbf{r},b_{1}+b_{2}} = \lambda_{\textbf{r},b_{1}} + \lambda_{\textbf{r}, b_{2}}$. Similarly, one can deduce that $\lambda_{\textbf{r},cb}= c\lambda_{\textbf{r},b}$.
    
\end{proof}
        The following important proposition shows that $\lambda_{\textbf{r},b} = \lambda_{\textbf{s}, b} $
\begin{proposition}
    There exists an algebra homomorphism $\psi: B \to \mathbb{C}$ such that $t^{\textbf{r}}b \cdot v = \psi(b)t^{\textbf{r}}\cdot v$ for every $\textbf{r}\neq \textbf{0}$.
    
\end{proposition}
\begin{proof}
    Note that for any $\textbf{r}\neq 0$ and $b\in B$, using Lemma 3.3, the operator $t^{\textbf{r}}b-\lambda_{\textbf{r},b} t^{\textbf{r}}=t^{\textbf{r}}(b-\lambda_{\textbf{r}, b}1)$ is locally nilpotent on $V$. Theorem 3.5 implies that $A'\otimes(b-\lambda_{\textbf{r},b}1)b'\cdot V =0$ and hence \begin{equation*}t^{\textbf{s}}(bb' -\lambda_{\textbf{r},b})b'.v =0, \forall \:\textbf{0}\neq \textbf{s}, b,b'\in B, \: v \in V. \end{equation*}                              Then, \begin{equation}
        t^{\textbf{s}}bb' = \lambda_{\textbf{r}, b}t^{\textbf{s}}b' \; \text{on} \;  V,\; \forall b,b'\in B, \textbf{r},\textbf{s}\in \mathbb{Z}^{N}\setminus\{\textbf{0}\}.
    \end{equation}
    In particular, for $\textbf{s}=\textbf{r}$ and $b'=1$, we have 
    
    \begin{equation}t^{\textbf{r}}b\cdot v =\lambda_{\textbf{r},b}t^{\textbf{r}}\cdot v, \; \forall \:  \textbf{r} \in \mathbb{Z}^{N}\setminus\{\textbf{0}\}, b\in B,\:  v\in V. \end{equation} 
    Using (3.7.13) and (3.7.14), we get
    \begin{equation}
        \lambda_{\textbf{s},bb' }t^{\textbf{s}}\cdot v = \lambda_{\textbf{r},b}\lambda_{\textbf{s},b'}t^{\textbf{s}}\cdot v, \; \forall \; b,b'\in B, \textbf{r},\textbf{s}\in \mathbb{Z}^{N}\setminus\{\textbf{0}\}, v\in V.
    \end{equation}\
    Injectivity of $t^{\textbf{s}}$ gives that $ \lambda_{\textbf{s},bb'}= \lambda_{\textbf{r},b}\lambda_{\textbf{s},b'}$, on putting $b'=1$ and using the fact that $\lambda_{\textbf{s},1}=1$, using (3.7.14) and (3.7.15), we obtain 
    \begin{equation} \lambda_{\textbf{r},b}=\lambda_{\textbf{s},b}, \: \forall \:  \textbf{r}, \textbf{s} \in  \mathbb{Z}^{N} \setminus \{\textbf 0\},\end{equation} Therefore, we have shown that the linear maps $\lambda_{\textbf{r},b}$ are the same for all $\textbf{r}\neq \textbf{0}$, and we use the notation $\psi(b)$ for $\lambda_{\textbf{r},b}$.  Finally, we get that 
    \begin{equation}
        \psi(bb')=\psi(b)\psi(b').
    \end{equation}
\end{proof}

Now we consider the composites of the operators $t^{\textbf{r}}b$ while acting on $V$ and consider the quasi-associativity of the $A\otimes B$-action on $V$. Let $\lambda_{\textbf{r},\textbf{s}, b, b'}$ be the eigenvalue of matrix $\textbf{B}_{\textbf{m},\textbf{n},b,b'} $ defined by \begin{equation*}
    t^{\textbf{m}}b\cdot t^{\textbf{n}}b'\cdot (v_{i})_{i=1}^{m} = (v_{i}(\textbf{m}+\textbf{n}))_{i=1}^{m}\textbf{B}_{\textbf{m},\textbf{n},b,b'}.
    \end{equation*}
It is evident that $\textbf{B}_{\textbf{m},\textbf{n},b,b'}= \begin{cases}
    \psi(bb')\textbf{B}_{\textbf{m},\textbf{n}}, \;\; \; \; \; \; \; \;   \text{if}\; \; \textbf{m},\textbf{n} \neq \textbf{0} \\
    \psi(b)\phi(b') \textbf{B}_{\textbf{m},\textbf{n}}, \; \; \; \text{if} \;  \textbf{m}\neq 0 \: \text{and} \: \textbf{n}=\textbf{0}, \\
    \phi(b)\phi(b')\textbf{B}_{\textbf{m},\textbf{n}}, \; \; \;   \text{if} \; \; \textbf{m}=\textbf{n}=\textbf{0},
    
\end{cases}$ \\
where $\textbf{B}_{\textbf{m},\textbf{n}}=\textbf{B}_{\textbf{m},\textbf{n},1,1}$. If $\lambda_{\textbf{m},\textbf{n}}=\lambda_{\textbf{m},\textbf{n},1,1}$, then $\lambda_{\textbf{m},\textbf{n},b,b'}= \begin{cases}
    \psi(bb')\lambda_{\textbf{m},\textbf{n}}, \; \text{if} \; \textbf{m},\textbf{n}\neq \textbf{0},  \\
    \psi(b)\phi(b')\lambda_{\textbf{m}, \textbf{n}}, \; \textbf{} \; \textbf{m}\neq \textbf{0}, \textbf{n}= \textbf{0},  \; 
    \\
    \phi(b)\phi(b')\lambda_{\textbf{m},\textbf{n}}, \; \text{if} \; \textbf{m}=\textbf{n}=\textbf{0}. \\
    
\end{cases}$ \\
Now we have the following theorem:
\begin{theorem} $t^{\textbf{m}}b\cdot t^{\textbf{n}}b' = \lambda_{\textbf{m},\textbf{n},b,b'}t^{\textbf{m}+\textbf{n}}$ on $V$, for every $ \textbf{m},\textbf{n}\in \mathbb{Z}^{N}$.
\end{theorem}
\begin{proof} By Engel's theorem, there exists a non-zero vector $v\in V_{\Lambda}$ such that $(t^{\textbf{m}}b\cdot t^{\textbf{n}}b^\prime- \lambda_{\textbf{m},\textbf{n},b , b^\prime}t^{\textbf{m}+\textbf{n}})\cdot v = 0$. Note that $W= \{ v \in V \mid (t^{\textbf{m}}b\cdot t^{\textbf{n}}b^\prime -\lambda_{\textbf{m}, \textbf{n},b , b'}t^{\textbf{m}+\textbf{n}})\cdot v =0, \forall \textbf{m},\textbf{n}\in \mathbb{Z}^{N}, b,b^\prime \in B \}$ is a non-zero $(S_{N}\ltimes A_{N})\otimes B$ submodule of $V$, and due to the irreducibilty of $V$, it follows that $W=V$.
\end{proof}
The following Theorem due to Rao is also useful for our purposes, it essentially establishes the relation between various $\lambda_{\textbf{m},\textbf{n}}, \textbf{m},\textbf{n}\in \mathbb{Z}^{N}$. \begin{theorem}{[\textbf{[14]}, Theorem 9.1]} Suppose $A_{N}'$ acts non-trivally on $V$ and $t^{\textbf{0}}$ acts by some non-zero scalar on $ V$, then there exists non-zero scalars $\lambda_{\textbf{m},\textbf{n}},\: \textbf{m},\textbf{n}\in \mathbb{Z}^{N}$ such that $ t^{\textbf{n}}t^{\textbf{n}}=\lambda_{\textbf{m},\textbf{n}}t^{\textbf{m},\textbf{n}}$ on $ V$ satisyfing \begin{enumerate}
    \item $\lambda_{\textbf{m},\textbf{n}}=\lambda$, if $\textbf{m},\textbf{n},\textbf{m}+\textbf{n}\in \mathbb{Z}^{N}\setminus\{\textbf{0}\}$.
    \item $\lambda_{-\textbf{m},\textbf{m}}=\mu$, if $\textbf{m}\in \mathbb{Z}^{N}\setminus \{\textbf{0}\}$. 
    \item $\lambda_{\textbf{m},\textbf{0}}=c$ and $\lambda^{2}=\mu c$.
\end{enumerate}
    
\end{theorem}

    \begin{proof}
         Obtained similarly for the $\mathcal{S}_{N}$ case, where the operators $h_{\textbf{r}}$ for the $\mathcal{H}_{N}$ in $[14]$ replaced by the operators  $d_{ij}(\textbf{r})$, for $1\leq i\neq j \leq N$.
    \end{proof}
\subsection{Action of $\mathcal{S}_{N}\otimes B$ }
Let $V$ be a nonzero irreducible uniformly bounded weight module for $(\mathcal{S}_{N}\ltimes A_{N})\otimes B$.
Let $\lambda \in \mathfrak{h}^{*}$ be such that $V_{\lambda}\neq 0$. It follows from the irreducibility of $V$ that $V=\bigoplus_{\textbf{r}\in \mathbb{Z}^{n}}V_{\Lambda + \textbf{r}}$, with $V_{\Lambda +\textbf{r}} = \{v \in V \mid d_{i}\cdot v = (\lambda(d_{i})+r_{i})v, \forall \:  1\leq i \leq N \}$, where $\Lambda = (\lambda(d_{i}))_{i=1}^{N}$. Also, since $t^{\textbf{r}}$ acts injectively on $V$, it follows that $t^{\textbf{r}}\cdot V_{\Lambda}=V_{\Lambda+\textbf{r}}$. We can therefore identify $V$ with $V_{\Lambda}\otimes A_{N}$.
Let $\mathcal{U}$ denote the universal enveloping algebra of $(\mathcal{S}_{N}\ltimes A_{N})\otimes B$. Then $\mathcal{U}$ is a $\mathbb{Z}^{N}$-graded Lie algebra and $\mathcal{U}= \bigoplus_{\textbf{n} \in \mathbb{Z}^{N}}\mathcal{U}_{n}$, where $\mathcal{U}_{\textbf{n}}= \{ u \in \mathcal{U} \: \mid \: [d_{i}, u]=n_{i}u, \: \forall\:  1\leq i \leq N\}$. Let $\mathcal{L}$ be the two-sided ideal in $\mathcal{U}$, generated by elements of the form $t^{\textbf{r}}b-\psi(b)t^{\textbf{r}}$($\textbf{r}\neq \textbf{0}$), $t^{\textbf{0}}b - c\phi(b)$ and $ t^{\textbf{r}}b_{1}t^{\textbf{s}}b_{2}- \lambda_{\textbf{r},\textbf{s},b_{1},b_{2}} t^{\textbf{r}+\textbf{s}}$($\forall\;  \textbf{r},\textbf{s}\in \mathbb{Z}^{N}$). Since $A_{N}\otimes B$ acts quasi-associatively on $V$,  $ \mathcal{L}$ acts as zero on $V$. Therefore, $V$ is an irreducible $\mathcal{U}/\mathcal{L}$-module. \par
We define the elements $T(u,\textbf{r},b_{1},b_{2})=t^{-\textbf{r}}b_{1}D(u,\textbf{r})b_{2}\in \mathcal{U}/\mathcal{L}$, for $u\in \mathbb{C}^{N}, \textbf{r} \in \mathbb{Z}^{N}, b_{1}, b_{2}\in B$, with $(u,\textbf{r})=0$, and let 
\begin{equation*}
    D= \text{span}_{\mathbb{C}}\{ T(u,\textbf{r}, b_{1}, b_{2}) \mid u\in \mathbb{C}^{N}, \textbf{r}\in \mathbb{Z}^{N}, b_{1}, b_{2}\in B \: \text{and} \:  (u,\textbf{r})=0\}.
\end{equation*}
\begin{lemma}
    \begin{enumerate}
        \item $D$ is a Lie algebra with the Lie bracket  \begin{multline*}
            [T(u,\textbf{r},b_{1},b_{2}),T(v,\textbf{s},b_{3},b_{4}) ] = T(w,\textbf{r}+\textbf{s},b_{1}b_{3},b_{2}b_{4})  - (u,\textbf{s})T(v,s,b_{1}b_{2}b_{3},b_{4}) \\ + (v,\textbf{r})T(u,\textbf{r},b_{1}b_{3}b_{4},b_{2}),
        \end{multline*}
        where $w=(u,\textbf{s})v-(v,\textbf{r})u$.
        \item $V_{\Lambda}$ is an irreducible $D$-module.
    \end{enumerate}
    \end{lemma}
    \begin{proof}
        \begin{enumerate}
            \item The proof is the same as that of [2]. The only thing to verify here is that $(w,\textbf{r}+\textbf{s})=0$. By definition, $w=(u,\textbf{s})v-(v,\textbf{r})u$, so that $(w,\textbf{r}+\textbf{s})= (u,\textbf{s})(v,\textbf{r})-(v,\textbf{r})(u,\textbf{s})=0$.
            \item  It is easy to see that $V_{\Lambda}$ is a $D$-module. Let $v,w\in V_{\Lambda}$ be two arbitrary vectors. Irreducibility of $V_{\Lambda}$ together with weight argument gives $X\in (\mathcal{U}/\mathcal{L})_{0}$ such that $X\cdot v = w$. Also, PBW theorem implies that $(\mathcal{U}/\mathcal{L})_{0}= \mathcal{U}(D)$. Therefore, it follows that $V_{\Lambda}$ is an irreducible $D$-module.
        \end{enumerate}
    \end{proof}
For $u\in \mathbb{C}^{N}$, $\textbf{r}\in \mathbb{Z}^{N}\setminus\{\textbf{0}\}$, we define $ T_{1}(u,\textbf{r},b_{1},b_{2})=t^{-\textbf{r}}b_{1}D(u,\textbf{r})b_{2}-cD(u,0)b_{1}b_{2} $ and Lie algebra
\begin{equation*}
    D_{1}=\text{span}_{\mathbb{C}}\{T_{1}(u,\textbf{r},b_{1},b_{2}) \mid u\in \mathbb{C}^{N}, \textbf{r} \in \mathbb{Z}^{N}, b_{1},b_{2}\in B \: \text{and} \: (u,\textbf{r})=0,\textbf{r}\neq \textbf{0}\}.
\end{equation*}
Also, define $I(u,\textbf{r},b_{1},b_{2})= \psi(b_{1})D(u,\textbf{r})b_{2}-cD(u,0)b_{1}b_{2}$ and Lie algebra \begin{equation*}
    D_{2}= \text{span}_{\mathbb{C}}\{I(u,\textbf{r},b_{1},b_{2}) \mid u \in \mathbb{C}^{N}, \textbf{r}\in \mathbb{Z}^{N}, b_{1},b_{2} \in B \: \text{and } (u,\textbf{r})=0, \textbf{r}\neq \textbf{0} \:  \} .
    \end{equation*}
Furthermore, we define the subspace $W$ of $V$, defined by 
\begin{equation*}
    W= \text{span}_{\mathbb{C}}\{t^{\textbf{r}}\cdot v - cv \mid v\in V_{\Lambda}, \textbf{r} \in \mathbb{Z}^{N} \}=  \text{span}\{t^{\textbf{r}}\cdot v - cv \mid v\in V, \textbf{r} \in \mathbb{Z}^{N} \}.
    \end{equation*}
Note that $\text{dim}(V/W)= \text{dim}(V_{\Lambda})$.
\begin{proposition}
    \begin{enumerate}
        \item $D_{1}$ is a Lie algebra with the bracket \begin{multline*}
            [T_{1}(u,\textbf{r},b_{1},b_{2}), T_{1}(v,s,b_{3},b_{4})] = T_{1}(w,\textbf{r}+\textbf{s},b_{1}b_{3},b_{2}b_{4})-(u,\textbf{s})T_{1}(v,\textbf{s},b_{3},b_{1}b_{2}b_{4}) \\
            +(v,r)T_{1}(u,\textbf{r}, b_{1}, b_{2}b_{3}b_{4}),
        \end{multline*}
        where $w=(u,\textbf{s})v-(v,\textbf{r})u$. Also, $[D,D] \subseteq D_{1}$.
        \item  $D_{2}$ is a Lie subalgebra of $\mathcal{S}_{N}'\otimes B$ with the Lie bracket 
        \begin{multline*}
            [I(u,\textbf{r},b_{1},b_{2}), I(v,\textbf{s},b_{3},b_{4})]= I(w,\textbf{r}+\textbf{s},b_{1}b_{3},b_{2}b_{4})-(u,\textbf{s})I(v,\textbf{s}, b_{3},b_{1}b_{2}b_{4})
+(v,\textbf{r})I(u,\textbf{r},b_{1},b_{2}b_{3}b_{4}),        \end{multline*}
where $w=(u,\textbf{s})v-(v,\textbf{r})u$.
\item The mapping $\pi: D_{1}\to D_{2}$ defined by $ \pi(T_{1}(u,\textbf{r},b_{1},b_{2}))= I(u,\textbf{r}, b_{1}, b_{2})$ is a surjective Lie algebra homomorphism. 
  \item $V_{\Lambda}$ is an irreducible $D_{1}$-module.
    \item $W$ is $D_{2}$-module. In particular, $V/W$ is a $D_{2}$-module and a $D_{1}$-module via $\pi$.
    \item $V_{\Lambda}\cong V/W$ as $D_{1}$-modules.
    \end{enumerate}
   
\end{proposition}
\begin{proof}
    The proof is the same as that of [9] on Loop-Witt algebras.
\end{proof}
As define in [9], we consider the space
\begin{equation*}
    \tilde{D}= \text{span}\{\psi(b_{1})D(u,\textbf{r})b_{2} - D(u,\textbf{r})b_{1}b_{2} \mid \textbf{r} \in \mathbb{Z}^{N}\setminus\{\textbf{0}\}, u \in \mathbb{C}^{N}, b_{1},b_{2} \in B \: \text{and} \; (u, \textbf{r})=0 \}.
\end{equation*}
It is evident that $\tilde{D}\subseteq D_{2}$ and $\tilde{D}= \mathcal{S}_{N}'\otimes M$, where $M$ is the maximal ideal $\text{ker}\psi$. Since $\mathcal{S}_{N}'$ is an ideal of $\mathcal{S}_{N}$ and $M$ is an ideal of $B$,  it follows that $\tilde{D}$ is an ideal of $\mathcal{S}_{N}\otimes B$, it is therefore an ideal of $D_{2}$. Now, we proceed to prove the following important lemma, but the proof incorporated here is different from that in [9], these methods cannot be used because $(u,\textbf{r})=0$.
\begin{lemma}
  D(u,\textbf{r})b - $\psi(b)D(u,\textbf{r})$ acts as scalars on $V_{\Lambda}$ for all $u\in \mathbb{C}^{N}$ and $\textbf{r}\in \mathbb{Z}^{N}\setminus\{0\}$ with $(u,\textbf{r})=0$.
  
\end{lemma}
\begin{proof}
     By Proposition 3.10, $V/W$ is an irreducible $D_{2}$-module. Let $\varphi: D_{2}\to \mathfrak{gl}(V/W) $ be the above Lie algebra homomorphism arising from the $D_{2}$-module structure on $V/W$ and $\rho = \varphi|_{\tilde{D}}$. We consider the following two possibilities.
    \begin{enumerate}
        \item If $M^{k}=0$ for some $k\in \mathbb{N}$. \\
        We know that $[\mathcal{S}_{N}'\otimes M, \mathcal{S}_{N}'\otimes M]= \mathcal{S}_{N}'\otimes M^{2}$, $[\mathcal{S}_{N}'\otimes M^{i}, \mathcal{S}_{N}'\otimes M^{j}]= \mathcal{S}_{N}'\otimes M^{i+j}(i,j\geq 1)$. This implies that $\tilde{D}=\mathcal{S}_{N}'\otimes M$ is a nilpotent ideal in $D_{2}$, hence $\varphi(\tilde{D})$  is a solvable ideal in $D_{2}$. As a consequence of Lemma 2.1, we see that $\varphi(D_{2})$ is a reductive Lie algebra and $\varphi(\tilde{D})$ being a solvable ideal of the reductive algebra $\varphi(D_{2})$, act as scalars on $V/W$. As a consequence, the operators $D(u,\textbf{r})b- \psi(b)D(u,\textbf{r})$ act as scalars on $V/W$. Therefore $D(u,\textbf{r})b - \psi(b)D(u,\textbf{r}) =\lambda(u,\textbf{r},b)$ on $V/W$, where $\lambda(u,\textbf{r},b)\in \mathbb{C}$.
\item If $M^{k}\neq 0$ for all $k\in \mathbb{N}$. Then it is enough to show that $\mathcal{S}_{N}'\otimes M^{k}$ act trivially  on $V/W$ for some $k\in \mathbb{N}$. Because $[(\mathcal{S}_{N}\otimes M)/\text{ker}\varphi, (\mathcal{S}_{N}'\otimes M^{i})/\text{ker}\varphi]=(\mathcal{S}_{N}'\otimes M^{i+1})/\text{ker}\varphi$ for all $i\in \mathbb{N}$, which gives that $(\mathcal{S}_{N}\otimes M)/\text{ker}\varphi$ is a solvable ideal of the reductive Lie algebra $D_{2}/\text{ker}\varphi$. Therefore $\varphi(\tilde{D})$ act as scalars on $V/W$.  \par 
We now proceed to show that $\mathcal{S}_{N}'\otimes M^{k}$ acts trivially on $V/W$ for some $k\in \mathbb{N}$. 
\par
\textbf{Claim 1:} $D(u,\textbf{r})\otimes M^{n}\subset J$, for some $u\in \mathbb{C}^{N}\setminus\{\textbf{0}\}$, $\textbf{r} \in \mathbb{Z}^{N}\setminus\{\textbf{0}\}$ with $(u,\textbf{r})=0$ and $n\in \mathbb{N}$. 

\par Let $d= \text{dim}((\mathcal{S}_{N}\otimes M)/J)$, where $J=\text{ker}\rho$.
    For any $ b_{1}\in M$, consider the following set  \begin{equation*} \{D(u_{i},\textbf{r}_{i})b_{1}+J \mid  u_{i} \in \mathbb{C}^{N}\setminus\{0\}, \textbf{r}_{i} \in \mathbb{Z}^{N}\setminus \{0\}, (u_{i},\textbf{r}_{i})=0 \: \text{and} \: \textbf{r}_{i} \notin \mathbb{Q}\textbf{r}_{j} \: \forall \: i\neq j \in \mathbb{Z}  \}.\end{equation*}
    It is evident that the above set is an infinite set. As $\text{dim}(\mathcal{S}_{N}\otimes M)/J=d$, there is a nonzero element $X= \sum_{i=1}^{k} D(u_{i},\textbf{r}_{i})b_{1} \in J$ for some $k\leq d+1$.
    For any $ 0\neq v\in \textbf{r}_{1}^\perp$ and $b_{2}\in M$, consider the following commutator:
    \begin{equation*}
      X'=  [X, D(v,-\textbf{r}_{1})b_{2}]=\sum_{i=2}^{k} D(-(u_{i},\textbf{r}_{1})v-(v,\textbf{r}_{i})u_{i},\textbf{r}_{i}-\textbf{r}_{1})b_{1}b_{2}.
    \end{equation*}
    Our objective is to ensure that the first term in the RHS of the expression above is nonzero, so we choose $v\in \textbf{r}_{1}^{\perp}$ such that $u_{2}'=-(u_{2},\textbf{r}_{1})v-(v,\textbf{r}_{2})u_{2}\neq 0$. We shall demonstrate the existence of such a vector $v$.  
    \par If $(u_{2},\textbf{r}_{1})=0$, then we choose $v\in \textbf{r}_{1}^{\perp}$ but $v\notin \textbf{r}_{2}^{\perp}$, such a choice is always possible as $\textbf{r}_{1} \notin \mathbb{Q}\textbf{r}_{2} $. If $(u_{2},\textbf{r}_{1})\neq 0$, then for any $0\neq v \in \textbf{r}_{1}^{\perp}$, it is easy to see that $-(u_{2},\textbf{r}_{1})v-(v,\textbf{r}_{2})u\neq 0 $. Therefore, in either case , we obtain $v\in \textbf{r}_{1}^{\perp}$ such that $ u_{2}'=-(u_{2},\textbf{r}_{1})v-(v,\textbf{r}_{2})u_{2}\neq 0$. If $u_{i}'=-(u_{i},\textbf{r}_{1})v-(v,\textbf{r}_{i})u_{i}=0$ for all $ i\geq 3 $, then we are done. If $u_{i}'\neq 0$ for some $i\geq 3$. Then it follows that $X'$ is the sum of atmost $k-1$ terms and for some $v'\in \textbf{r}_{1}^{\perp}$, we observe that
    \begin{equation*}
        X^{[2]}=[X',D(v', \textbf{r}_{1})b_{3}]= \sum_{i=2}^{k} D(u_{i},\textbf{r}_{i})b_{1}b_{2}b_{3}.
    \end{equation*}
        We remark that the above expression involves a slight abuse of notation, as we have again used the notation $u_{i}$ to denote the vectors $(u_{i}',r_{1})v' -(v',r_{i})u_{i}'$, for the sake of clarity. Here we again choose $v'\in \textbf{r}_{1}^{\perp }$ such that $u_{2}=(u_{2}',\textbf{r}_{1})v' -(v,\textbf{r}_{2})u_{2}'\neq 0$. If $u_{i}=0$ for all $i\geq 3$, then we are done. If this is not the case, we proceed analogously to the initial steps, we get that $D(u,\textbf{r})b_{1}b_{2}b_{3}\ldots b_{n}\in J$, for some $u\in \mathbb{C}^{N}\setminus \{0\}, \textbf{r} \in \mathbb{Z}^{N}\setminus \{0\}$, any $b_{1},\ldots,b_{n}\in M$ and $n\leq d+1$.  \\
    \textbf{Claim 2:} $\mathcal{S}_{N}'\otimes M^{n+2}\subset J$. 
    
\textbf{Subclaim 2.1:}    \par If $ \textbf{t}\in \mathbb{Z}^{N}$ is such that $(u,\textbf{t})\neq 0$, then we can find $\textbf{s}\in \mathbb{Z}^{N}\setminus\{\textbf{0}\} $ and $ v\in \textbf{s}^{\perp}$ such that $[D(u,\textbf{r}), D(v,\textbf{s})]= D(w,\textbf{t})$, for any $\textbf{0}\neq w\in \textbf{t}^{\perp}$. \\We now prove the existence of such $\textbf{s}\in \mathbb{Z}^{N}\setminus\{\textbf{0}\}$ and $\textbf{0}\neq v\in \textbf{s}^{\perp}$. Simplifying the commutator in the equation above, we have $ D((u,\textbf{s})v-(v,\textbf{r})u,\textbf{r}+\textbf{s})=D(w,\textbf{t})$. Consequently, the problem of determining a solution $D(v,\textbf{s}) \in \mathcal{S}_{N}'$ is equivalent to finding $\textbf{s}\in \mathbb{Z}^{N}\setminus\{\textbf{0}\}$ and $v\in \textbf{s}^\perp$ such that $(u,\textbf{s})v-(v,\textbf{r})u= w$ and $\textbf{r}+\textbf{s}=\textbf{t}$. \par 
    
    Take $\textbf{s}= \textbf{t}-\textbf{r}$, then $(u,\textbf{r})=0$ and $(u,\textbf{t})\neq 0$ implies $(u,\textbf{s})\neq 0$. We consider the following two cases for $(w,\textbf{s})$ : either $(w,\textbf{s})\neq 0 $ or $(w,\textbf{s})=0$. If $(w,\textbf{s})\neq 0$, then there exists $0\neq a\in \mathbb{C}  $, such that $(w,\textbf{s})=a(u,\textbf{s})$, then take $v=\frac{1}{(u,\textbf{s})}(w-au)$. Note that $0\neq v\in \textbf{s}^{\perp}$ and $(u,\textbf{s})v-(v,\textbf{r})u =w$ and  $\textbf{r}+\textbf{s}=\textbf{t}$. Again, if $(w,\textbf{s})=0$, then $(w,\textbf{t})=0$ gives that $(w,\textbf{s})=0$, so that $0 \neq w\in \textbf{s}^{\perp}\cap \textbf{r}^{\perp}$(for $N\geq 3$), in this case we take $v=\frac{1}{(u,\textbf{s})}w$, and we are done. Further, observe that when $ N=2$, the case $(w,\textbf{s})\neq 0$ cannot occur for any nonzero $w$, as $ \textbf{r}\notin \mathbb{Q}\textbf{s}$, and $(w,\textbf{s})=0$ implies $(w,\textbf{s})=0$ and hence $w=\textbf{0}$, which is not the case.   In the case $(u,\textbf{t})\neq 0$, $D(v,\textbf{s})$ is the solution to the equation above, where $v= \frac{1}{(u,\textbf{s})}(w-au)$ and  $\textbf{s}=\textbf{t}-\textbf{r}$. Consequently, for any $ b_{n+1},b_{n+2}\in M$, we have
    \begin{equation}
        [D(u,\textbf{r})b_{1}\cdots b_{n},D(v,\textbf{s})b_{n+1}b_{n+2}] = D(w,\textbf{t})b_{1}\cdots b_{n+2} \in J, \; \text{for} \; (u,\textbf{t})\neq 0.
    \end{equation}If $(u,\textbf{t})=0$, then the above methods can no longer be utilized; but however, we can move the points $u\in \mathbb{C}^{N}( \text{resp.} \: \textbf{r}\in \mathbb{Z}^{N})$ to $u'\in \mathbb{C}^{N}(\text{resp.}\: \textbf{r}'\in \mathbb{Z}^{N})$ such that $(u', \textbf{t})\neq 0$, so that above method can then be applied. Thus, it remains to find $\textbf{s}\in \mathbb{Z}^{N}\setminus \{0\}$, and $\textbf{0}\neq v \in \textbf{s}^{\perp}$ such that $(u',\textbf{t})\neq 0$, where $u'=(u,\textbf{s})v-(v,\textbf{r})u$. Note that for any $\textbf{s} \in \mathbb{Z}^{N}\setminus\{\textbf{0}\}$ and $\textbf{0}\neq v \in \textbf{s}^{\perp}$, we have:
         \begin{align*}
            (u',\textbf{s}')= (u',\textbf{t}-\textbf{r}')&=((u,\textbf{s})v-(v,\textbf{r})u, \textbf{t} ), \; \text{since } \textbf{r}'=\textbf{r}+\textbf{s},\\
             &=(u,\textbf{s})(v,\textbf{t})-(v,\textbf{r})(u,\textbf{t}) \\
             &= (u,\textbf{s})(v,\textbf{t}), \; \text{since} \; (u,\textbf{t})=0
             \end{align*}

             Therefore, showing that $(u',\textbf{s}')\neq 0$ is equivalent to showing that  $(u,\textbf{s})(v,\textbf{t})\neq 0$. We choose $\textbf{s}\in \mathbb{Z}^{N}$, such that $(u,\textbf{s})\neq 0$ and $v\in \textbf{s}^{\perp}$ such that $v \notin \textbf{t}^{\perp}$, i.e. $(v,\textbf{t})\neq 0$. Thus, we have $(u,\textbf{s})(v,\textbf{t})\neq 0$.
            For such $\textbf{s}\in \mathbb{Z}^{N}$ and $v\in \textbf{s}^{\perp}$, we have 
            \begin{equation*}
                [D(u,\textbf{r}), D(v,\textbf{s})]= D(u',\textbf{r}'),
            \end{equation*}
            where $u'\in \mathbb{C}^{N}\setminus \{0\}$ and $\textbf{r}'\in \mathbb{Z}^{N}\setminus\{0\}$ are such that $(u',\textbf{t})\neq 0$. As $(u',\textbf{t})\neq 0$, we can use Subclaim 2.1 to find $v'\in \textbf{s}^{\perp} $, such that $[D(u',\textbf{r}'), D(v',\textbf{s}')] = D(w,\textbf{t})$. Then for any $b_{n+1},b_{n+2}\in M$, we have  \begin{equation}
                [[D(u,\textbf{r})b_{1}\ldots b_{n}, D(v,\textbf{s})b_{n+1}],D(v',\textbf{s}')b_{n+2}]=D(w,\textbf{s})b_{1}\ldots b_{n+1}b_{n+2}.
            \end{equation}
         From (3.12.18) and (3.12.19), it follows that $\mathcal{S}_{N}'\otimes M^{n+2}\subseteq J$.

      \end{enumerate}.
\end{proof}
By Lemma 3.12, we have $D(u,\textbf{r})b = \psi(b)D(u,\textbf{r}) + \lambda(u,\textbf{r},b)$ on $V/W$, where $\lambda(u,\textbf{r},b)$ is a scalar. Note that $D(u,\textbf{r})b-\psi(b)D(u,r)=D(u,\textbf{r})b-cD(u,0)b - (\psi(b)D(u,\textbf{r})-cD(u,0)b)$, it is easy to see that its preimage under the homomorphism $\pi: D_{1} \to D_{2}$ is $ t^{-r}D(u,\textbf{r})b-cD(u,0)b -(t^{-\textbf{r}}bD(u,\textbf{r})-cD(u,0)b)= t^{-\textbf{r}}D(u,\textbf{r})b-t^{-\textbf{r}}bD(u,\textbf{r})$, therefore the elements  $t^{-\textbf{r}}D(u,\textbf{r})b-t^{-\textbf{r}}bD(u,\textbf{r})$ act as scalars,say $\lambda(u,\textbf{r}, b)$ on $V/W$. Thus, we have \begin{equation}
t^{-\textbf{r}}D(u,\textbf{r})b = \psi(b)t^{-\textbf{r}}D(u,\textbf{r}) + \lambda(u,\textbf{r},b) \: \: \text{on $V/W$}.
\end{equation}
By the isomorphism of $D_{1}$-modules $V_{\Lambda}$ and $V/W$, it follows that $t^{-\textbf{r}}D(u,\textbf{r})b-\psi(b)t^{-\textbf{r}}D(u,\textbf{r})$ act as scalar, say $\lambda(u,\textbf{r},b)$ on $V_{\Lambda}$. Note that $\lambda(u,\textbf{r},b)$ is linear bilinear in $u$ and $b$, and that $\lambda(u,\textbf{r},1)=0$.
Now we consider the irreducible $D$-module structure on $V_{\Lambda}$ to understand the action of $D(u,\textbf{0})b$ on $V_{\Lambda}$. Consider \begin{align*}
    [D(v,0)b, t^{-\textbf{r}}D(u,\textbf{r})b']&= t^{-\textbf{r}}[D(v,0)b, D(u,\textbf{r})b']+ [D(v,0)b,t^{-\textbf{r}}]D(u,\textbf{r})b', \\
    & = (v,\textbf{r})t^{-\textbf{r}}D(u,\textbf{r})bb' -(v,\textbf{r})t^{-\textbf{r}}bD(u,\textbf{r})b'.
\end{align*}
Therefore, for any $w\in \mathbb{C}^{N}$ with $(w,\textbf{r})\neq 0$ and $v\in V_{\Lambda}$, we have
\begin{align*}
    [D(w,0)b,t^{-\textbf{r}}D(u,\textbf{r})b']\cdot v = (w,\textbf{r})(\psi(bb')t^{-\textbf{r}}D(u,\textbf{r}) + \lambda(u,\textbf{r},bb'))\cdot v \\ -(w,\textbf{r})\psi(b)(\psi(b')t^{-\textbf{r}}D(u,\textbf{r})+ \lambda(u,\textbf{r},b'))\cdot v \\ =(w,\textbf{r})(\lambda(u,\textbf{r},bb')-\psi(b)\lambda(u,\textbf{r},b'))\cdot v .
    \end{align*}
Note that since $(w,\textbf{r})\neq 0$, and $ [D(w,\textbf{0}),t^{-\textbf{r}}D(u,\textbf{r})b']\in \mathfrak{sl}(V_{\Lambda})$, we must have $\lambda(u,\textbf{r},bb')-\psi(b)\lambda(u,\textbf{r},b')=0$, i.e. $\lambda(u,\textbf{r},bb')=\psi(b)\lambda(u,\textbf{r},b') $, for every $ b,b'\in B$, $u\in \mathbb{C}^{N}$, $\textbf{r}\in \mathbb{Z}^{N}$ with $(u,\textbf{r})=0$. In particular, for $b'=1$, we have $\lambda(u,\textbf{r},b)=\psi(b)\lambda(u,\textbf{r},1)=0$. Therefore, we have $[D(w,\textbf{0})b, t^{-\textbf{r}}D(u,\textbf{r})b']=0$.

\begin{corollary} The elements D(u,0)b act as scalars on $V_{\Lambda}$.
\end{corollary}
\begin{proof}
    Use Schur's lemma for the irreducible $D$-module $V_{\Lambda}$.
\end{proof}
Let $D(u,\textbf{0})b$ acts by the scalar $ f(u,b)$, then $f(u,b)$ is linear in both of its arguments and $f(u,1)=(u,\Lambda)$. As a consequence of Cor. 3.13, we see that the elements $D(u,0)b $ act as scalars($f(u,b)+(u,\textbf{r})\psi(b)$) on each graded component $V_{\Lambda+\textbf{r}}$. 
\begin{corollary}
    The elements $D(u,\textbf{r})b$ act as $\psi(b)D(u,\textbf{r})$ in $V_{\Lambda}$.
\end{corollary}
\begin{proof}
We have shown that $\lambda(u,\textbf{r},b)=0$, therefore (3.12.20) becomes
\begin{equation}
    t^{-\textbf{r}}D(u,\textbf{r})b\cdot v =\psi(b)t^{-\textbf{r}}D(u,\textbf{r})\cdot v, \: \: \text{for $v \in V_{\Lambda}$}.
\end{equation}
Operating on both sides of (3.14.21) by $t^{r}$ and using the fact $t^{\textbf{r}}t^{-\textbf{r}}=\lambda_{-\textbf{r},\textbf{r}}t^{0}=c\lambda_{-\textbf{r},\textbf{r}}$ on $V$, we obtain 
\begin{equation*}
    D(u,\textbf{r})b=\psi(b)D(u,\textbf{r}) \; \text{on} \; V_{\Lambda}.
\end{equation*}

\end{proof}
The following corollary provides an immediate proof of this equality throughout module $V$.
\begin{corollary}
    For any $v\in V$, we have $D(u,\textbf{r})b\cdot v=\psi(b)D(u,\textbf{r})\cdot v$ for every $u\in \mathbb{C}^{N}, \textbf{r}\in \mathbb{Z}^{N}\setminus\{\textbf{0}\}\}$ with $(u,\textbf{r})=0$.
\end{corollary}
\begin{proof}
    If is enough to show the result for any homogeneous graded component. Let $v'\in V_{\Lambda +\textbf{r}}=t^{\textbf{r}}V_{\Lambda}$, then $v'=v\otimes t^{\textbf{r}}$ for some $v\in V_{\lambda}$. Consider \begin{align*}D(u,\textbf{s})b.v'&= D(u,\textbf{s})b.t^{\textbf{r}}.v, \\ &=[D(u,\textbf{s})b,t^{\textbf{r}}]\cdot v + t^{r}D(u,\textbf{s})b\cdot v, \\ &= (u,\textbf{r})t^{\textbf{r}+\textbf{s}}(b) \cdot v + \psi(b)t^{\textbf{r}}D(u,\textbf{s})\cdot v, \\ &= \psi(b)(u,\textbf{r})t^{\textbf{r}+\textbf{s}}\cdot v + \psi(b)t^{\textbf{r}}D(u,\textbf{s})\cdot v \\ &= \psi(b)D(u,\textbf{s})\cdot t^{\textbf{r}}\cdot v,  \\ 
    &= \psi(b)D(u,\textbf{s})\cdot v'.\end{align*}
\end{proof}
We shall now proceed to show that $V$ is an irreducible weight module for the underlying Lie algebra $(\mathcal{S}_{N}\ltimes A_{N})$. For any $0\neq v \in V_{\Lambda + \textbf{r}}$, using Prop. 3.7 and Cor. 3.15, we have \begin{equation}
    \mathcal{U}((\mathcal{S}_{N}\ltimes A_{N})\otimes B)\cdot v=\mathcal{U}(\mathcal{S}_{N}\ltimes A_{N})\cdot v.\end{equation} If we show that this equality also holds for any $0\neq v\in V$, then we are done. For this purpose, we prove the following important lemma. 
    \begin{lemma}
        For any $0\neq v \in V$, we have $\mathcal{U}(\mathcal{S}_{N}\ltimes A_{N})\cdot v = V$.
        
    \end{lemma}
    \begin{proof}
        It is sufficient to prove the result for $v=v_{1}+v_{2}\in V$ with $v_{1}\in V_{\Lambda +\textbf{r}}$ and $v_{2}\in V_{\Lambda + \textbf{s}}$, for $\textbf{r}\neq\textbf{s}\in \mathbb{Z}^{N}$. It suffices to prove that $\mathcal{U}((\mathcal{S}_{N}\ltimes A_{N})\otimes B)\cdot v \subseteq \mathcal{U}(\mathcal{S}_{N}\ltimes A_{N})\cdot v$; to this end, we show that for homogeneous elements $x_{i}$'s of $(S_{N}\ltimes A_{N})\otimes B$, we have $x_{1}x_{2}\ldots x_{n}\cdot v\in \mathcal{U}(S_{N}\ltimes A_{N})$. For homogeneous elements of $(S_{N}\ltimes A_{N})\otimes B$ of nonzero degree $\textbf{r}\neq \textbf{0}$, namely, elements of the form $D(u,\textbf{r})b$ or  $t^{\textbf{r}}b$ for some $\textbf{r}\in \mathbb{Z}^{N}\setminus\{\textbf{0}\}$, we have $D(u,\textbf{r})b =\psi(b)D(u,\textbf{r})$ and $t^{\textbf{r}}b =\psi(b)t^{\textbf{r}}$ on $V$. Therefore, it is sufficient to show that $x_{1}\ldots x_{n}\cdot v \in \mathcal{U}(S_{N}\ltimes A_{N})\cdot v$, for the zero degree elements $x_{i}$ of $(\mathcal{S}_{N}\ltimes A_{N})\otimes B$. Furthermore, since the elements $t^{\textbf{0}}b$ act as scalars $c\phi(b)$ on $V$, therefore, we only consider the zero-degree elements of the form $D(u,\textbf{0})b $. We have $D(u,\textbf{0})b\cdot v = (f(u,b) + (u,\textbf{r})\psi(b) )v_{1} + (f(u,b)+(u,\textbf{s})\psi(b))v_{2} = f(u,b)(v_{1}+v_{2})+\psi(b)((u,\textbf{r})v_{1}+(u,\textbf{s})v_{2})$. Since $D(u,\textbf{0})\cdot v_{1}=(u,\Lambda + \textbf{r})v_{1}$, therefore, we have $D(u,\textbf{0})b\cdot v = f(u,b)(v_{1}+v_{2})+\psi(b)D(u,\textbf{0})\cdot (v_{1}+v_{2}) - (u,\Lambda)(v_{1}+v_{2})$. Note that $f(u,b)-(u,\Lambda)$ is a scalar, and $D(u,\textbf{0})\in S_{N}\ltimes A_{N}$, therefore $D(u,\textbf{0})b\cdot v \in \mathcal{U}(\mathcal{S}_N \ltimes A_{N})\cdot v$. Since the action of degree-zero elements preserves the grading, it does not alter the degree of the graded vectors.  The process can be therefore be repeated to prove that $x_{1}\cdots x_{n}\cdot v \in \mathcal{U}(S_{N}\ltimes A_{N})\cdot v $. The PBW theorem then implies that 
        $\mathcal{U}((\mathcal{S}_{N}\ltimes A_{N})\otimes B)\cdot v\subseteq \mathcal{U}(S_{N}\ltimes A_{N})\cdot v. $\end{proof}

 Lemma 3.16 implies $V$ is an irreducible weight module for $\mathcal{S}_{N}\ltimes A_{N}$ with finite-dimensional weight spaces, where the action of $A_{N}$ is nontrivial and quasi-associative (using Theorem 3.8). Therefore, $V$ is an irreducible elements of the Category $\mathcal{J}$ for the Lie algebra $\mathcal{S}_{N}$. Such modules are classified by Y. Billig and J. Talboom in [1]. 
We now proceed to the following theorem on the structure of $V$.
\begin{theorem}
    Let $V$ be an irreducible uniformly bounded weight module for $(\mathcal{S}_{N}\ltimes A_{N})\otimes B$, with the condition that $t^{\textbf{s}}$ and $t^{\textbf{0}}$ act non-trivially on $V$, for some $0\neq \textbf{s}\in \mathbb{Z}^{N}$, then $V$ can be written as $V=\bigoplus_{\textbf{r}\in \mathbb{Z}^{N}}V_{\Lambda +\textbf{r}}$, where $V_{\Lambda+\textbf{r}}=\{ v\in V\mid D(u,\textbf{0})v =(u,\Lambda +\textbf{r})v \}$ and $\Lambda =(\lambda(d_{i}))_{i=1}^{N}\in \mathbb{C}^{N}$ for any $\lambda \in P(V)$. The elements  $t^{\textbf{r}}$ act injectively on $V$ for every $\textbf{r}\in \mathbb{Z}^{N}$ and the set of weights of $V$ is $P(V)=\lambda +\sum_{i=1}^{N}\mathbb{Z}\delta_{i}$. Moreover, there exists an algebra homomorphism $\psi: B \to \mathbb{C}$ with $\psi(1)=1$, a linear map $\phi:B \to \mathbb{C}$ with $\phi(1)=1$ and a bilinear function $f: \mathbb{C}^{N}\times B \to \mathbb{C}$ with $f(u,1)=(u,\Lambda)$ such that \begin{align*}
        t^{\textbf{r}}(b) &= \psi(b)t^{\textbf{r}} \; \text{on} \; V, \; \text{for} \; \textbf{r}\neq \textbf{0}, \\ 
        t^{\textbf{0}}(b) &=\phi(b)t^{\textbf{0}} \; \text{on} \: V, \; \\ 
        D(u,\textbf{r})b &= \psi(b)D(u,\textbf{r}) \; \text{on $V$, for $\textbf{r}\neq 0$},\\
       D(u,\textbf{0})b\cdot v &= (f(u,b)+(u,\textbf{r})\psi(b))v, \; \text{for} \: v\in V_{\Lambda +\textbf{r}}, \textbf{r}\in \mathbb{Z}^{N}.
    \end{align*}
    Furthermore, the action of $A_{N}\otimes B$ on $V $ is quasi-associative in the following sense :
    \begin{equation*}
        t^{\textbf{m}}(b)t^{\textbf{n}}(b') = \begin{cases}
            t^{\textbf{m}+\textbf{n}}(bb')= \psi(bb')\lambda t^{\textbf{m}+\textbf{n}}, \; \; \text{if} \; \textbf{m},\textbf{n}, \textbf{m}+\textbf{n}\in \mathbb{Z}^{N}\setminus \{0\}, \\
            \psi(bb')\mu t^{0}, \: \text{if} \:  \textbf{m}, \textbf{n} \in \mathbb{Z}^{N}\setminus \{0\}, \textbf{m}+\textbf{n}=\textbf{0}, \\
            \psi(b)\phi(b')ct^{\textbf{m}}, \: \text{if}
 \: \textbf{m}\neq \textbf{0}, \textbf{n}=\textbf{0}, \\
 \phi(b)\phi(b')ct^{\textbf{0}}, \:  \text{if} \: \textbf{m}=\textbf{n}=\textbf{0},\end{cases}
  \end{equation*}
 on V, where $\lambda, \mu$ are nonzero complex numbers with $\lambda^{2}=c\mu$.  In addition, for any $0\neq v\in V$, we have \begin{equation} V= \mathcal{U}((\mathcal{S}_{N}\ltimes A_{N})\otimes B )\cdot v = \mathcal{U}(\mathcal{S}_{N}\ltimes A_{N})\cdot v .\end{equation}
    Therefore, $V$ is an irreducible weight module for $\mathcal{S}_{N}\ltimes A_{N}$ with finite-dimensional weight spaces, in particular $V$ is an irreducible element in the Category $\mathcal{J}$-modules for the Lie algebra $\mathcal{S}_{N}$.
\end{theorem}
\begin{proof}
Follows from Corollary 3.13, Corollary 3.14.
\end{proof}

\section{Harish-Chandra Modules for Extended Hamiltonian Lie algebras}
We first recall a result due to Talboom on irreducible modules of the category $\mathcal{J}$-modules of the Lie algebra $\mathcal{H}_{N}$.
\begin{theorem}{[\textbf{[30]}, Theorem 5.2]}
    Let $W$ be an irreducible Jet module for $\tilde{\mathcal{H}}_{N}$ having finite-dimensional weight spaces with respect to $\mathfrak{h}$. Then there exists a finite-dimensional irreducible module $V$ for $\mathfrak{sp}_{N}$ and $\alpha,\beta \in \mathbb{C}^{N}$ such that $W\cong V\otimes A_{N}$ with the following actions. \begin{enumerate}
        \item $D(u,\textbf{0})\cdot v\otimes t^{\textbf{k}}=(u,\alpha+\textbf{k})v\otimes t^{\textbf{k}}$.
        \item $h(\textbf{r})\cdot v\otimes t^{\textbf{s}} = (\overline{\textbf{r}},\textbf{s})v\otimes t^{\textbf{r}+\textbf{s}} + (\sum_{i}r_{m+i}\beta_{m+i} + \sum_{i}r_{i}\beta_{i})v\otimes t^{\textbf{r}+\textbf{s}}+ \sum_{i}( r_{m+i}^{2}E_{m+i,i}$ $+ r_{i}r_{m+i}(E_{ii}-E_{m+i,m+i})-r_{i}^{2}E_{i,m+i})v\otimes t^{\textbf{r}+\textbf{s}} + \sum_{i,j}(r_{m+i}r_{m+j}(E_{m+j,i}+E_{m+i,j})+ r_{i}r_{m+j}(E_{ij}-E_{m+j,m+i}) -r_{i}r_{j}(E_{i,m+j}+E_{j,m+i}))v\otimes t^{\textbf{r}+\textbf{s}}$.
        \item $t^{\textbf{r}}\cdot v\otimes t^{\textbf{s}}=v\otimes t^{\textbf{r}+\textbf{s}}$.
        \end{enumerate}
\end{theorem}
Let $V$ be an irreducible uniformly bounded Harish-Chandra module for $(\tilde{\mathcal{H}}\ltimes A_{N})\otimes B$ with finite-dimensional weight spaces with the assumption that $t^{\textbf{0}}$ and $t^{\textbf{r}}$ for some $\textbf{r}\neq \textbf{0}$ act nontrivially on $V
$. Since $t^{\textbf{0}}\otimes b$ is a central element of $(\tilde{\mathcal{H}}_{N}\ltimes A_{N})\otimes B $ for every $b\in B$ and $V$ is a uniformly bounded Harish-Chandra module, there exists a linear map $\phi: B \to \mathbb{C}$ such that $t^{\textbf{0}}\otimes b =\phi(b)t^{0}$ on $V$, and hence $t^{\textbf{0}}\otimes b=c\phi(b)$ on $V$, where $c$ is the nonzero scalar by which $t^{\textbf{0}}$ acts on $V$. For simplicity, we denote by $t^{\textbf{r}}b$ the elements $t^{\textbf{r}}\otimes b$, for every $\textbf{r}\in \mathbb{Z}^{N}$ and $ b \in B$.
\subsection{Quasi associative action of $A_{N}\otimes B$ }
We state the following lemmas, whose proofs are analogous to those for 
the case of $(\mathcal{S}_{N}\ltimes A_{N})\otimes B$. 
\begin{lemma}
    Let $V=\bigoplus_{\textbf{n}\in \mathbb{Z}^{N}}V_{\textbf{n}}$ be an irreducible uniformly bounded $\mathbb{Z}^{N}$-graded module over the associative algebra $A_{N}$, then $\text{dim}V_{\textbf{n}}\leq 1$, for all $\textbf{n}\in \mathbb{Z}^{N}$.
\end{lemma}
\begin{lemma}  Suppose that $V$ is an irreducible module for $(\tilde{\mathcal{H}}_{N}\ltimes A_{N})\otimes  B$ and $x\in \mathcal{U}(A_{N}\otimes B)$. If $x\cdot v=0$ for some nonzero $v\in V$, then x is locally nilpotent on V.
\end{lemma}
\begin{proposition}
    Let $V$ be an irreducible uniformly bounded Harish-Chandra module for $(\tilde{\mathcal{H}}_{N}\ltimes A_{N})\otimes B$. Then for $b\in B$ all $ t^{\textbf{m}}b,\;  \textbf{m}\neq 0$ act injectively on $V$ or all $t^{\textbf{m}}b,\: \textbf{m}\neq 0$ act nilpotently on V.
\end{proposition}
\begin{proof}
    Suppose that $t^{\textbf{m}}b$ acts injectively on $V$ for some $\textbf{m}\neq 0$. Let $\textbf{n}\in \mathbb{Z}^{N}$ be such that $(\overline{\textbf{m}}| \textbf{n})\neq 0$ and if $t^{\textbf{n}}b $ acts locally nilpotently on $V$, then we see that it is also nilpotent on $V$. As $t^{\textbf{n}}b$ is locally nilpotent on $V$, it follows that $t^{-\textbf{n}}bt^{\textbf{n}}b$ is also locally nilpotent on $V$. As $V$ is a uniformly bounded weight module, we can find a positive integer $N'$ such that $(t^{-\textbf{n}}bt^{\textbf{n}}b)^{N'}\cdot V=0$, i.e. $(t^{-\textbf{n}}b)^{N'}(t^{\textbf{n}}b)^{N'}\cdot V=0$. As $(\overline{\textbf{m}}|\textbf{n})\neq 0$, we see that for any $k\in \mathbb{Z}_{+}$, we have \begin{align*}
        h_{\textbf{m}+\textbf{n}}\cdot (t^{-\textbf{n}}b)^{k}\cdot (t^{\textbf{n}}b)^{k}\cdot V &= (t^{-\textbf{n}}b)^{k}\cdot (t^{\textbf{n}}b)^{k}\cdot h_{\textbf{m}+\textbf{n}}\cdot V + (\overline{\textbf{m}}|\textbf{n})kt^{\textbf{m}+2\textbf{n}}b\cdot (t^{-\textbf{n}}b)^{k}\cdot (t^{\textbf{n}}b)^{k-1}\cdot V \\
        &\quad - (\overline{\textbf{m}}|\textbf{n})kt^{\textbf{m}}b\cdot (t^{\textbf{n}}b)^{k}\cdot (t^{-\textbf{n}}b)^{k-1}\cdot V.
    \end{align*}
    In particular where $k=N'$, we obtain: \begin{equation*}
        N't^{\textbf{m}+2\textbf{n}}b\cdot (t^{-\textbf{n}}b)^{N'}\cdot (t^{\textbf{n}}b)^{N'-1}\cdot V- N't^{\textbf{m}}b\cdot (t^{\textbf{n}}b)^{N'}\cdot(t^{-\textbf{n}}b)^{N'-1}\cdot V=0
        \end{equation*}
Applying $t^{\textbf{n}}b$ on both sides of the equation above, we get: $t^{\textbf{m}}b\cdot (t^{\textbf{n}}b)^{N'+1}\cdot (t^{-\textbf{n}}b)^{N'-1}\cdot V=0$ and since $t^{\textbf{m}}b$ acts injectively on $V$, we get that $(t^{\textbf{n}}b)^{N'+1}(t^{-\textbf{n}}b)^{N'-1}\cdot V=0$. By reverse induction, it follows that $(t^{\textbf{n}}b)^{N_{0}}\cdot V =0$, for some $N_{0}\in \mathbb{Z}_{+}$. This shows that $t^{\textbf{n}}b$ is locally nilpotent on $V$. \par
Our next step is to show that $t^{\textbf{n}}b$ is injective for all $\textbf{n}\in \mathbb{Z}^{N}\setminus\{\textbf{0}\}$.
\\ 
\textbf{Case 1:} We first prove this result for $\textbf{n}\in \mathbb{Z}^{N}\setminus\{\textbf{0}\}$ with $(\overline{\textbf{m}}|\textbf{n})\neq 0$. To derive a contradiction, assume that $\textbf{n}\in \mathbb{Z}^{N}\setminus\{\textbf{0}\}$ with $ (\overline{\textbf{m}},\textbf{n})\neq 0 $ be such that $t^{\textbf{n}}b$ is locally nilpotent on $V$, then by preceding argument, there exists $k \in \mathbb{Z}_{+}$ such that $(t^{\textbf{n}}b)^{k}\cdot V=0$. As $V$ is uniformly bounded, we can assume that $\text{dim}(V_{\Lambda})\geq \text{dim}(V_{\Lambda +\textbf{r}})$, for every $\textbf{r}\in \mathbb{Z}^{N}$. By the injectivity of $t^{\textbf{m}}b$, we have $t^{\textbf{m}}b\cdot V_{\Lambda + i \textbf{m}}= V_{\Lambda + (i+1)\textbf{m}}$ and $\text{dim}V_{\Lambda + i\textbf{m}}= \text{dim}V_{\Lambda }$, for every $i \in \mathbb{Z}$. For $v\in V_{\Lambda-k\textbf{m}}$, we  have 
\begin{equation*}
    0= h_{\textbf{m}-\textbf{n}}^{k}\cdot (t^{\textbf{n}}b)^{k} \cdot v= k!(\overline{\textbf{m}}|\textbf{n})^{k}(t^{\textbf{m}}b)^{k}\cdot v + t^{\textbf{n}}b\cdot u,
\end{equation*}
for some $u \in V_{\Lambda -\textbf{n}}$, therefore $t^{\textbf{n}}b : V_{\Lambda-\textbf{n}}\to V_{\Lambda}$ is bijective, and by the same arguments as used in Claim 3 of Proposition 3.4, it follows $t^{\textbf{n}}b:V_{\Lambda -(i+1)\textbf{n}}\to V_{\Lambda - i\textbf{n}} $ is bijective for every $i\in \mathbb{N}$, and  therefore it cannot be nilpotent on $V$, a contradiction.  \\ 
 \textbf{Case 2:} It remains to be shown that $t^{\textbf{n}}b$ is injective in $V$ for $\textbf{n}\in \mathbb{Z}^{N}\setminus\{\textbf{0}\}$ with $(\overline{\textbf{m}}|\textbf{n})=0$. For such a $\textbf{n}$, we have $(\overline{\textbf{m}}| \textbf{n}+\overline{\textbf{m}})=(\overline{\textbf{m}}|\overline{\textbf{m}})\neq 0$, and by Case 1, $t^{\textbf{n}+\overline{\textbf{m}}}b$ is injective. This also gives that $t^{\textbf{m}}b$ is injective if and only if $t^{\overline{\textbf{m}}}b$ is injective. Therefore, $t^{\textbf{n}+\overline{\textbf{m}}}b$ is injective implies that $t^{\overline{\textbf{n}}-\textbf{m}}b$ is injective. Note that $(\overline{\overline{\textbf{n}}-\textbf{m}},\textbf{n})= -(\textbf{n},\textbf{n})\neq 0$, so that $t^{\textbf{n}}b$ is injective on $V$.
\end{proof}
\begin{lemma}
    If the element $t^{\textbf{r}} b$ acts nilpotently on $V$ for some $\textbf{r}\neq 0$, then $t^{\textbf{s}}bB\cdot V=0$ for every $\textbf{s}\in \mathbb{Z}^{N}\setminus\{\textbf{0}\}$ and hence $A'\otimes bB\cdot V=0$, where $bB$ is the principal ideal $<b>$ of $B$ generated by $b$. 
\end{lemma}
\begin{proof}
    If $t^{\textbf{r}}b$ is locally nilpotent for some $\textbf{r}\neq \textbf{0}$, then $t^{\textbf{n}}b$ is locally nilpotent for each $\textbf{n}\in \mathbb{Z}^{N}\setminus\{\textbf{0}\}$. We can deduce (3.5.10), (3.5.12) analogously to Theorem 3.5, with the operators $d_{i,j}(\textbf{m})'s$ replaced by the operators $h_{\textbf{m}} $ and the additional assumption that $(\overline{\textbf{m}_{i}}, \textbf{n})\neq 0$. Therefore, we have 
    \begin{equation*}
        t^{\textbf{m}_{1}-\textbf{n}}bb_{1}\cdots t^{\textbf{m}_{3N}-\textbf{n}}bb_{3N}\cdot V =0, \forall \textbf{m}_{i} \in \mathbb{Z}^{N} \: \:  \text{with } \: \; 
        \sum_{i=1}^{3N}\epsilon_{i}\textbf{m}_{i} \notin \mathbb{Q}\textbf{n} \; \text{and} \: (\overline{\textbf{m}}_{i},\textbf{n})\neq 0 \; \forall \; i.
    \end{equation*}
    Now for $\textbf{n}_{i}\in \mathbb{Z}^{N}$, with $1\leq i \leq 3N$, let $\textbf{n}\in \mathbb{Z}\setminus \sum_{i=1}^{3N}\epsilon_{i}\textbf{n}_{i}$, $\forall \epsilon_{i}\in \{0,1\}$ with $(\overline{\textbf{n}},\textbf{n}_{i})\neq 0$, for all $1\leq i \leq 3N$. Let $\textbf{m}_{i}=\textbf{n}+\textbf{n}_{i}$, then $(\overline{\textbf{m}}_{i},\textbf{n})\neq 0$, it follows that \begin{equation*}
        t^{\textbf{n}_{1}}bb_{1}\cdots t^{\textbf{n}_{3N}}bb_{3N}\cdot V =0.
    \end{equation*}
    This is true for any $\textbf{n}_{i}\in \mathbb{Z}^{N}\setminus\{\textbf{0}\}$ $(1\leq i \leq 3N)$. Now for the above $b$, $W=\{v\in V \mid t^{\textbf{m}}bB\cdot v=0 \; \forall \textbf{m}\in \mathbb{Z}^{N}\setminus\{\textbf{0}\} \}$ is nonzero $(\mathcal{\tilde{H}}_{N}\ltimes A_{N})\otimes B$-submodule of $V$. By the irreducibility of $V$, we conclude that $W=V$.
\end{proof}
By our assumption, $A_{N}'$ acts nontrivially, therefore $t^{\textbf{r}}$ is injective for all $\textbf{r}\neq 0$. Proceeding similarly to the case of $(\mathcal{S}_{N}\ltimes A_{N})\otimes B$, we see that $\text{dim}V_{\textbf{n}}=\text{dim}V_{\textbf{m}}, \forall \textbf{m},\textbf{n}\in \mathbb{Z^{N}} $. For each $0\neq \textbf{r}\in \mathbb{Z}^{N}$, there are scalars $ \lambda_{\textbf{r},b}$ such that $t^{\textbf{r}}b =\lambda_{\textbf{r},b}t^{\textbf{r}}$ on $V$, so that the mapping $b\to \lambda_{\textbf{r},b}$ defines a linear map from $B \to \mathbb{C}$ for every fixed $\textbf{r}\neq 0$. The following important proposition proves that such a map is in fact independent of $\textbf{r}$.
\begin{proposition}
    There exists an algebra homomorphism $\psi: B \to \mathbb{C}$ such that $t^{\textbf{r}}b =\psi(b)t^{\textbf{r}}$ on $V$, for every $\textbf{r}\neq 0$.
\end{proposition}

We observe the following quasi-associative action of $A_{N}\otimes B$ on $V$.
\begin{equation}
    t^{\textbf{m}}b\cdot t^{\textbf{n}}b'= \lambda_{\textbf{m},\textbf{n},b,b'}t^{\textbf{m}+\textbf{n}} \: \: \text{on $V$},
\end{equation}
 where $\lambda_{\textbf{m},\textbf{n},b,b'} =\begin{cases}
     \psi(bb')\lambda, \:  \text{ if} \:  \textbf{m},\textbf{n},\textbf{m}+\textbf{n} \in \mathbb{Z}^{N}\setminus\{\textbf{0}\}, \\
     \psi(bb')\mu, \; \text{if} \; \textbf{m}.\textbf{n} \in \mathbb{Z}^{N}\setminus\{\textbf{0}\}, \textbf{m}+\textbf{n}=\textbf{0}  \\
     \psi(b)\phi(b')c, \: \text{if} \; \textbf{m}\neq \textbf{0}, \textbf{n}=\textbf{0}, \\
     \phi(b)\phi(b')c, \; \text{if} \; \textbf{m}=\textbf{n}=\textbf{0}.
 \end{cases}$ 
 \\
 This follows from arguments analogous to those Theorem 3.8 for $\mathcal{S}_{N}$ and Theorem 3.9.
 \subsection{Action of $\mathcal{\tilde{H}}_{N}\otimes B$}
 We now investigate the action of $\mathcal{\tilde{H}}_{N}\otimes B$ on $V$. Let $\lambda \in \mathfrak{h}^*$ be a weight of $V$. From the irreducibility of $V$, it follows that $V=\bigoplus_{\textbf{r}\in \mathbb{Z}^{N}}V_{\Lambda +\textbf{r}}$, where $\Lambda=(\lambda(d_{i}))_{i=1}^{N}$. Injectivity of $t^{\textbf{r}}$ impliess $t^{\textbf{r}}V_{\Lambda } = V_{\Lambda + \textbf{r}}$ for every $\textbf{r}\in \mathbb{Z}^{N}$. Therefore, we identify $V$ with $V_{\Lambda}\otimes A_{N}$. Let $\mathcal{U}$ be the universal enveloping algebra of $(\tilde{\mathcal{H}}_{N}\ltimes A_{N})\otimes B$, then $\mathcal{U}$ is a Lie algebra of $\mathbb{Z}^{N}$-graded Lie algebra and $\mathcal{U}=\bigoplus_{\textbf{n}\in \mathbb{Z}^{N}}\mathcal{U}_{\textbf{n}}$, where  $\mathcal{U}_{\textbf{n}}=\{ u\in \mathcal{U} \mid [d_{i},u]= n_{i}u, \: \forall \: 1\leq i \leq N\}$. Let $\mathcal{I}$ be the two-sided ideal of $\mathcal{U}$ generated by elements of the form $t^{\textbf{r}}b-\psi(b)t^{\textbf{r}}(\textbf{r}\neq \textbf{0}), t^{\textbf{0}}b-\phi(b)c$, $t^{\textbf{r}}bt^{\textbf{s}}b' -\lambda_{\textbf{r},\textbf{s},b,b'}t^{\textbf{r}+\textbf{s}}$. By the quasi-associative action of $A_{N}\otimes B$ on $V$, as observed above, it follows that $V$ is an irreducible $\mathcal{U}/\mathcal{I}$-module. \\
  Let  $T(\overline{\textbf{r}},\textbf{r},b_{1},b_{2})= t^{-\textbf{r}}b_{1}D(\overline{\textbf{r}},\textbf{r})b_{2}=\psi(b_{1})t^{-\textbf{r}}D(\overline{\textbf{r}},\textbf{r})b_{2}\in \mathcal{U}/\mathcal{I}$ and $T(u,\textbf{0},b_{1},b_{2})= t^{\textbf{0}}b_{1}D(u,0)b_{2}$ $  \in \mathcal{U}/\mathcal{I}$ and define the Lie algebra $D$ as follows. \begin{equation*}
     D= 
     \text{span}_{\mathbb{C}}\{T(\overline{\textbf{r}},\textbf{r},b_{1},b_{2}), T(u,\textbf{0},b_{1},b_{2}) \mid \textbf{r}\in \mathbb{Z}^{N}, u \in \mathbb{C}^{N}, b_{1},b_{2}\in B\}.
 \end{equation*}
 We have the following lemma.
 \begin{lemma}
     \begin{enumerate}
         \item D is a Lie algebra with the Lie bracket \begin{equation*}
             [T(\overline{\textbf{r}},\textbf{r},b_{1},b_{2}), T(\overline{\textbf{s}},\textbf{s}, b_{3},b_{4})]= (\overline{\textbf{r}},\textbf{s})T(\overline{\textbf{r}}+\overline{\textbf{s}},\textbf{r}+\textbf{s}, b_{1}b_{3},b_{2}b_{4}),
         \end{equation*}
         and 
         \begin{equation*}
             [T(u,\textbf{0},b_{1},b_{2}), T(\overline{\textbf{r}},\textbf{r},b_{3},b_{4})]=(u,\textbf{r})c\phi(b_{1})(-T(\overline{\textbf{r}},\textbf{r},b_{2}b_{3},b_{4})+ T(\overline{\textbf{r}},\textbf{r}, b_{3},b_{2}b_{4})).
             \end{equation*}
             \item $V_{\Lambda}$ is an irreducible $D$-module.
     \end{enumerate}
     \begin{proof}
         The proof is the same as that of $(\mathcal{S}_{N}\ltimes A_{N})\otimes B$, which is done earlier in Section 3.
     \end{proof}
 \end{lemma}
 We now define the following elements $T_{1}(\overline{\textbf{r}},\textbf{r},b_{1},b_{2})= t^{-\textbf{r}}b_{1}D(\overline{\textbf{r}},\textbf{r})b_{2}-cD(\overline{\textbf{r}}, \textbf{0})b_{1}b_{2} \in \mathcal{U}/\mathcal{I}$ and $I(\overline{\textbf{r}},\textbf{r},b_{1},b_{2})= \psi(b_{1})D(\overline{\textbf{r}},\textbf{r})b_{2}- cD(\overline{\textbf{r}},\textbf{0})b_{1}b_{2}\in \mathcal{\tilde{H}}_{N}\otimes B$ and define the Lie algebras $D_{1}$ and $D_{2}$ as follows:
 \begin{equation*}
     D_{1}=\text{span}_{\mathbb{C}}{\mathbb{}}\{ T_{1}(\overline{\textbf{r}},\textbf{r},b_{1},b_{2}) \mid \textbf{r}\in \mathbb{Z}^{N}, b_{1},b_{2}\in B \},
 \end{equation*}
 and 
 \begin{equation*}
     D_{2}= \text{span}_{\mathbb{C}}\{ I(\overline{\textbf{r}},\textbf{r},b_{1},b_{2}) \mid  \textbf{r}\in \mathbb{Z}^{N},b_{1},b_{2}\in B\},
 \end{equation*}
 which are the Lie subalgebras of $\mathcal{U}/\mathcal{I}$(resp. $\mathcal{\tilde{H}}_{N}\otimes B$) generated by $T_{1}(\overline{\textbf{r}},\textbf{r},b_{1},b_{2})$(resp. $I(\overline{\textbf{r}},\textbf{r},b_{1},b_{2})$).\\
 Define the subspace $W$ of $V$ by
 \begin{equation*}
     W=\text{span}_{\mathbb{C}}\{t^{\textbf{r}}\cdot v - cv\mid v\in V_{\Lambda}, \textbf{r}\in \mathbb{Z}^{N}\}.
 \end{equation*}
 Then $\text{dim}(V_{\Lambda})=\text{dim}(V/W)$.
 We have the following proposition.
 \begin{proposition}
     \begin{enumerate}
         \item $D_{1}$ is a Lie algebra with the bracket:\begin{align*}
             [T_{1}(\overline{\textbf{r}},\textbf{r},b_{1},b_{2}), T_{1}(\overline{\textbf{s}},\textbf{s},b_{3},b_{4} )]=(\overline{\textbf{r}},\textbf{s})T_{1}(\overline{\textbf{r}}+\overline{\textbf{s}},\textbf{r}+\textbf{s},b_{1}b_{3},b_{2}b_{4} ) \\-(\overline{\textbf{r}},\textbf{s})T_{1}(\overline{\textbf{s}},\textbf{s}, b_{3},b_{1}b_{2}b_{4})+(\overline{\textbf{s}},\textbf{r})T(\overline{\textbf{r}},\textbf{r}, b_{1},b_{2}b_{3}b_{4} ),
             \end{align*}
         and $[D,D]\subset D_{1}$.
         \item $D_{2}$ is a Lie subalgebra of $\mathcal{\tilde{H}}_{N}\otimes B$ with the bracket:
         \begin{align*}
          [I(\overline{\textbf{r}},\textbf{r},b_{1},b_{2}), I(\overline{\textbf{s}},\textbf{s},b_{3},b_{4} )]=(\overline{\textbf{r}},\textbf{s})I(\overline{\textbf{r}}+\overline{\textbf{s}},\textbf{r}+\textbf{s},b_{1}b_{3},b_{2}b_{4} ) \\-(\overline{\textbf{r}},\textbf{s})I(\overline{\textbf{s}},\textbf{s}, b_{3},b_{1}b_{2}b_{4})+(\overline{\textbf{s}},\textbf{r})I(\overline{\textbf{r}},\textbf{r}, b_{1},b_{2}b_{3}b_{4} ),
         \end{align*}
         \item The mapping $\pi:D_{1}\to D_{2}$ defined by $\pi(T_{1}(\overline{\textbf{r}},\textbf{r}.b_{1},b_{2}))=I(\overline{\textbf{r}},\textbf{r}.b_{1},b_{2}) $ is a surjective Lie algebra homomorphism.
         \item $V_{\Lambda}$ is an irreducible $D_{1}$-module.
         \item $W$ is a $D_{2}$-module. In particular $V/W$ is a $D_{2}$-module and hence a $D_{1}$ module via $\pi$.
         \item $V_{\Lambda}\cong V/W$ as $D_{1}$-modules.
     \end{enumerate}
 \end{proposition}
 \begin{proof}
     The proof is analogous to that of $(\mathcal{S}_{N}\ltimes A_{N})\otimes B$.
 \end{proof}
 As before, we define the following space.
 \begin{equation*}
     \tilde{D} = \text{span}_{\mathbb{C}}\{ \psi(b_{1})D(\overline{\textbf{r}},\textbf{r})b_{2}-D(\overline{\textbf{r}},\textbf{r})b_{1}b_{2} \mid \textbf{r} \in \mathbb{Z}^{N},b_{1},b_{2}\in B \}.
 \end{equation*}
 Then $\tilde{D} \subseteq D_{2}$ and $\tilde{D}= \mathcal{H}_{N}\otimes M$, where $M=\text{ker}\psi$.
 We now prove the following important lemma.
 \begin{lemma}
    The elements $D(\bar{\textbf{r}},\textbf{r})b-\psi(b)D(\bar{\textbf{r}},\textbf{r})$ act as scalars on $V/W$.
 \end{lemma}
 \begin{proof}
     By Proposition 4.7. $V/W$ is an irreducible $D_{2}$-module. Let $\varphi: D_{2}\to \mathfrak{gl}(V/W)$ be the Lie algebra homomorphism arising from above irreducible module structure and $\rho= \varphi|_{\tilde{D}}$. If $M^{k}=0$ for some $k\in \mathbb{N}$, then the result follows by the argument used in Lemma 3.12. If $M^{k}\neq 0$, for every $k\in \mathbb{N}$. If show that $\mathcal{H}_{N}\otimes M^{k}$ acts trivially on $V/W$ for some $k\in 
     \mathbb{N}$, then the result will follows from the fact that $\varphi(\mathcal{H}_{N}\otimes M)$ is a solvable ideal of the reductive Lie algebra $\varphi(D_{2})$, therefore it acts by scalars on $V/W$ by Lemma 2.1. \\
     Let $d= \text{dim}(\mathcal{H}_{N}\otimes M)/J$, where $J=\text{ker}\rho$. We first show that $D(\overline{\textbf{r}},\textbf{r})\otimes M^{n}\subseteq J$, for some $\textbf{r}\in \mathbb{Z}^{N}$ and $n \in \mathbb{Z}_{+}$. Consider the following set 
     \begin{equation*}
         \{ D(\overline{\textbf{r}_{i}}, \textbf{r}_{i})b_{1}+ J \mid \textbf{r}_{i}\notin \mathbb{Q}\textbf{r}_{j}, \textbf{r}_{i} \notin \textbf{r}_{j}^\perp \}.
     \end{equation*}
    Note that the set described above is an infinite set. There exists non-zero scalars $c_{1},c_{2},..., c_{k}$($k\leq d+1$) such that $0\neq X = \sum_{i=1}^{k}c_{i}D(\overline{\textbf{r}_{i}},\textbf{r}_{i} )b_{1} \in J $. Then, \begin{equation*}
        X^{[1]}=[X,D(-\overline{\textbf{r}_{1}},-\textbf{r}_{1})b_{2}]=\sum_{i=2}^{k}c_{i}(\overline{\textbf{r}_{i}},\textbf{r}_{1})D(\overline{\textbf{r}_{i}} - \overline{\textbf{r}_{1}},\textbf{r}_{i} -\textbf{r}_{1})b_{1}b_{2}\in J
    \end{equation*}
    Taking the commutator of $ X^{[1]}$ with $D(\overline{\textbf{r}_{1}},\textbf{r}_{1})b_{3}$, we find:
    \begin{equation*}
        X^{[2]}=\sum_{i=2}^{k}c_{i}(\overline{\textbf{r}_{i}},\textbf{r}_{i})^{2}D(\overline{\textbf{r}_{i}},\textbf{r}_{i})b_{1}b_{2}b_{3}\in J.
    \end{equation*}\
    Inductively, it follows that $ X^{[k]}=D(\overline{\textbf{r}_{k}}-\overline{\textbf{r}_{k-1}}, \textbf{r}_{k}-\textbf{r}_{k-1})b_{1}b_{2}\ldots b_{k+1}\in J$. Thus, we find an element $D(\bar{\textbf{r}},\textbf{r})b_{1}b_{2}\ldots b_{2k-1}\in J$.
   \\ \textbf{Claim:} $\mathcal{H}_{N}'\otimes M^{k+3}\subseteq J$. 
   \\
   \textbf{Case 1:} For the $\textbf{r}$ above if  $\textbf{t}\in \mathbb{Z}^{N}\setminus\{\textbf{0}\}$ is such that $(\overline{\textbf{r}},\textbf{t})\neq 0$, then for $\textbf{s}=\textbf{t}-\textbf{r}$, we have: \begin{equation*}
        [D(\overline{\textbf{r}},\textbf{r})b_{1}\ldots b_{k+1}, D(\overline{\textbf{s}},\textbf{s})b_{k+2}b_{k+3}] = (\overline{\textbf{r}},\textbf{s})D(\overline{\textbf{t}},\textbf{t})b_{1}\ldots b_{k+3}\in J.
    \end{equation*}
    \textbf{Case 2:} If $\textbf{t}\in \mathbb{Z}^{N}$ is such that $(\overline{\textbf{r}},\textbf{t})=0 $, then we choose $\textbf{s}\in \mathbb{Z}^{N}$ such that $(\overline{\textbf{r}},\textbf{s})\neq 0$ and $(\overline{\textbf{r}}+\overline{\textbf{s}}, \textbf{t})\neq 0$, that is $(\overline{\textbf{s}},\textbf{t})\neq 0$. For $\textbf{r}'=\textbf{r}+\textbf{s}$, we can find $ \textbf{s}'\in \mathbb{Z}^{N}\setminus\{\textbf{0} \}$ such that $[D(\overline{\textbf{r}'},\textbf{r}'), D(\overline{\textbf{s}'},\textbf{s}')] =(\overline{\textbf{s}},\textbf{t})D(\overline{\textbf{t}},\textbf{t})$, follows from Case 1, given that $(\overline{\textbf{r}'},\textbf{t})\neq 0$. For such $\textbf{s}\in \mathbb{Z}^{N}\setminus\{\textbf{0}\}$, we take $ \textbf{s}'= \textbf{t}-\textbf{r}-\textbf{s}$, then \begin{equation*}
        [[D(\overline{\textbf{r}},\textbf{r})b_{1}\ldots b_{k+1},D(\overline{\textbf{s}},\textbf{s})b_{k+2}, D(\overline{\textbf{s}'},\textbf{s}')b_{k+3}]=(\overline{\textbf{r}},\textbf{s})(\overline{\textbf{s}},\textbf{t})D(\overline{\textbf{t}},\textbf{t})b_{1}\ldots b_{k+3} \in J.
    \end{equation*} 
    This completes the proof.
 \end{proof}
 As a consequence, the elements $ D(\overline{\textbf{r}},\textbf{r})b -\psi(b)D(\overline{\textbf{r}},\textbf{r})$ act as scalars on $V/W$.  Note that this element is the $\pi$-image of $t^{-\textbf{r}}D(\overline{\textbf{r}},\textbf{r})b-\psi(b)t^{-\textbf{r}}D(\overline{\textbf{r}},\textbf{r}) = t^{-\textbf{r}}D(\overline{\textbf{r}},\textbf{r})b - cD(\overline{\textbf{r}},\textbf{0})b -(t^{-\textbf{r}}bD(\overline{\textbf{r}}, \textbf{r})-cD(\overline{\textbf{r}},\textbf{0})b) \in D_{1} $, and hence $ \pi(t^{-\textbf{r}}D(\overline{\textbf{r}},\textbf{r})b-\psi(b)t^{-\textbf{r}}D(\overline{\textbf{r}},\textbf{r}))=D(\overline{\textbf{r}},\textbf{r})b -\psi(b)D(\overline{\textbf{r}},\textbf{r}) $. By the isomorphism of the $D_{1}$-modules $V_{\Lambda}$ and $V/W$, the elements $t^{-\textbf{r}}D(\overline{\textbf{r}},\textbf{r})b -\psi(b)t^{-\textbf{r}}D(\overline{\textbf{r}},\textbf{r}) $ act as scalars(say, $\lambda(\textbf{r},b)$) on $V_{\Lambda}$, i.e. \begin{equation}
     t^{-\textbf{r}}D(\overline{\textbf{r}},\textbf{r})b = \psi(b)t^{-\textbf{r}}D(\overline{\textbf{r}},\textbf{r})+\lambda(\textbf{r},b) \; \text{on} \; V_{\Lambda}.
 \end{equation} In what follows, we consider the action of the elements $D(u,\textbf{0})b$ on the irreducible $D$-module $V_{\Lambda}$. Note that $[D(u,\textbf{0})b,t^{-\textbf{r}}D(\overline{\textbf{r}},\textbf{r})b']=(u,\textbf{r})(-\psi(b)\lambda(\textbf{r},b')+ \lambda(\textbf{r},bb'))$ on $V_{\Lambda}$. Choosing $u\in \mathbb{C}^{N}$ such that $(u,\textbf{r})\neq 0$, then, since $[D(u,\textbf{0}),t^{-\textbf{r}}D(\overline{\textbf{r}},\textbf{r})]\in \mathfrak{sl}(V_{\Lambda})$ as an operator on $V_{\Lambda}$. By the zero trace condition of the operators $[D(u,\textbf{0})b,t^{-\textbf{r}}D(\overline{\textbf{r}},\textbf{r})b']$, we obtain: $\lambda(\textbf{r},bb')=\psi(b)\lambda(\textbf{r},b')=0$. In particular, we have $\lambda(\textbf{r},b)=0$, for all $b\in B$. 
 
 \par By Schur's lemma, the elements $D(u,\textbf{0})b$ act as scalars on $V_{\Lambda}$, say by $f(u,b)$, then it is easy to see that $f(u,1)=(u,\Lambda)$ and $f$ is bilinear
 . Note that $D(u,\textbf{0})b$ acts as $f(u,b)+\psi(b)(u,\textbf{r})$ on $V_{\Lambda+\textbf{r}}$, for each $\textbf{r}\in \mathbb{Z}^{N}$. Applying $t^{\textbf{r}}$ on both sides of (4.9.25
 ), we have: $D(\overline{\textbf{r}},\textbf{r})b=\psi(b)D(\overline{\textbf{r}},\textbf{r})$ on $V_{\Lambda}$. We have the following lemma. 
 \begin{lemma}
     The elements $D(\overline{\textbf{r}},\textbf{r})b$ act as $\psi(b)D(\overline{\textbf{r}},\textbf{r})$, for all $\textbf{r}\in \mathbb{Z}^{N}$ on $V$.
 \end{lemma}
 \begin{proof}
     It is sufficient to prove the equality on every graded component $V_{\Lambda +\textbf{s}}$, $\textbf{s}\in\mathbb{Z}^{N}$. Let $w\in V_{\Lambda+\textbf{s}}$, then $w=t^{\textbf{s}}\cdot 
     v$ for some $v\in V_{\Lambda}$. Consider $D(\overline{\textbf{r}},\textbf{r})b\cdot t^{\textbf{s}}\cdot v= t^{\textbf{s}}\cdot D(\overline{\textbf{r}},\textbf{r})b\cdot v + [D(\overline{\textbf{r}},\textbf{r})b, t^{\textbf{s}} ]\cdot v$, yielding $D(\overline{\textbf{r}},\textbf{r})b\cdot t^{\textbf{s}}\cdot v = \psi(b)t^{\textbf{s}}\cdot D(\overline{\textbf{r}},\textbf{r})\cdot v + (\overline{\textbf{r}},\textbf{s})t^{\textbf{r}+\textbf{s}}b\cdot v $, and hence $ D(\overline{\textbf{r}},\textbf{r})b\cdot t^{\textbf{s}}\cdot v= \psi(b)t^{\textbf{s}}\cdot D(\overline{\textbf{r}},\textbf{r})\cdot v +\psi (b)(\overline{\textbf{r}},\textbf{s})t^{\textbf{r}+\textbf{s}}\cdot v $. Therefore, we have $D(\overline{\textbf{r}},\textbf{r})b\cdot w=\psi(b)D(\overline{\textbf{r}},\textbf{r})\cdot w $, for $w\in  V_{\Lambda + \textbf{r}}$, for every $\textbf{r}\in \mathbb{Z}^{N}$. \end{proof}
We now show that $V$ is an irreducible weight module for the underlying Lie algebra $\mathcal{\tilde{H}}_{N}\ltimes A_{N}$ with finite-dimensional weight modules. We first prove the following proposition.
\begin{proposition}
    For any graded vector(nonzero) v of $V$, we have $\mathcal{U}((\mathcal{\tilde{H}}_{N}\ltimes A_{N})\otimes B)\cdot v = \mathcal{U}(\mathcal{\tilde{H}}_{N}\ltimes A_{N})\cdot v$.
\end{proposition}

Using Lemma 3.16(which can be proved similarly for $(\mathcal{\tilde{H}}_{N}\ltimes A_{N})\otimes B$) and Proposition 4.11, we conclude that for any nonzero vector $v$ of $V$, we have 
\begin{equation}
    \mathcal{U}(\mathcal{\tilde{H}}_{N}\ltimes A_{N})\cdot v = V.
\end{equation}
It follows that $V$ is an irreducible Harish-Chandra module for $(\mathcal{\tilde{H}}_{N}\ltimes A_{N})$ and the action of $A_{N}'$ and $t^{\textbf{0}}$ is nontrivial on $V$. In addition, the action of $A_{N}$ on $V$ can be verified to be quasi-associative; $ V$ is therefore an irreducible jet module for the Lie algebra $\mathcal{\tilde{H}}_{N}$. Such modules are classified by J. Talboom in [30]. We now proceed to state the main theorem concerning the structure of $V$.

\begin{theorem}
     Let $V$ be an irreducible uniformly bounded weight module for $(\tilde{\mathcal{H}}_{N}\otimes A_{N})\otimes B$, with the condition that $t^{\textbf{s}}$ and $t^{0}$ act non-trivially on $V$, for some $0\neq \textbf{r}\in \mathbb{Z}^{N}$, then $V$ be can be written as $V=\bigoplus_{\textbf{r}\in \mathbb{Z}^{N}}V_{\Lambda +\textbf{r}}$, where $V_{\Lambda +\textbf{r}}=\{v\in V \mid D(u,\textbf{0})\cdot v = (u,\Lambda +\textbf{r})v,\: \forall\; u \in \mathbb{C}^{N}\}$ and $\Lambda =(\lambda(d_{i}))_{i=1}^{N}$ for some $\lambda \in P(V)$. The elements $t^{\textbf{r}}$ act injectively on $V$ for every $\textbf{r}\in \mathbb{Z}^{N}$ and the set of weights of $V$ is given by $P(V)=\lambda + \sum_{i=1}^{N}\mathbb{Z}\delta_{i}$.  Moreover, there exists an algebra homomorphism $\psi: B \to \mathbb{C}$ with $\psi(1)=1$, a linear map $\phi:B \to \mathbb{C}$ with $\phi(1)=1$ and a bilinear function $f: \mathbb{C}^{N}\times B \to \mathbb{C}$ with $f(u,1)=(u,\Lambda)$ such that \begin{align*}
        t^{\textbf{r}}(b) &= \psi(b)t^{\textbf{r}} \; \text{on $V$, for} \; \textbf{r}\neq \textbf{0}, \\ 
        t^{\textbf{0}}(b) &=\phi(b)t^{0} \text{on } \: V, \\ 
        D(\overline{\textbf{r}},\textbf{r})b &= \psi(b)D(\overline{\textbf{r}},\textbf{r}) \; \text{on}\; V, \; \text{for} \; \textbf{r}\neq  \textbf{0} \\
       D(u,\textbf{0})b\cdot v &= (f(u,b)+(u,\textbf{r})\psi(b))v, \; \text{for} \: v\in V_{\Lambda +\textbf{r}}, \textbf{r}\in \mathbb{Z}^{N}.
    \end{align*}
    Also, the action of $A_{N}\otimes B$ is quasi-associative on $V$ in the following sense :
    \begin{equation*}
        t^{\textbf{m}}(b)t^{\textbf{n}}(b') = \begin{cases}
            t^{\textbf{m}+\textbf{n}}(bb')= \psi(bb')\lambda t^{\textbf{m}+\textbf{n}}, \; \; \text{if} \; \textbf{m},\textbf{n}, \textbf{m}+\textbf{n}\in \mathbb{Z}^{N}\setminus \{0\}, \\
            \psi(bb')\mu t^{0}, \: \text{if} \:  \textbf{m}, \textbf{n} \in \mathbb{Z}^{N}\setminus \{0\}, \textbf{m}+\textbf{n}=\textbf{0}, \\
            \psi(b)\phi(b')ct^{\textbf{m}}, \: \text{if}
 \: \textbf{m}\neq \textbf{0}, \textbf{n}=\textbf{0}, \\
 \phi(b)\phi(b')ct^{\textbf{0}}, \:  \text{if} \: \textbf{m}=\textbf{n}=\textbf{0}.\end{cases}
  \end{equation*}
    Furthermore, for any $0\neq v\in V$, we have \begin{equation}\mathcal{U}((\mathcal{\tilde{H}}_{N}\ltimes A_{N})\otimes B )\cdot v = \mathcal{U}({\mathcal{\tilde{H}}}_{N}\ltimes A_{N})\cdot v .\end{equation}
    Therefore, $V$ is an irreducible uniformly bounded module Harish-Chandra module for $\mathcal{\tilde{H}}_{N}\ltimes A_{N}$. In particular, $V$ is an irreducible jet module for the Lie algebra $\mathcal{\tilde{H}}_{N}$.
\end{theorem}


\end{document}